\documentclass[preprint,1p]{elsarticle}
\usepackage{amsmath}
\usepackage{amssymb}
\usepackage{graphicx}
\graphicspath{{./}}
\usepackage{xspace}
\usepackage{multirow}
\usepackage{mathbbol}
\usepackage{algorithm}
\usepackage{algpseudocode}
\usepackage{graphicx}
\usepackage{subcaption}
\usepackage{enumitem}

\setlist{topsep=12pt}
\setlist{itemsep=-3pt}

\usepackage{multirow}
\usepackage{adjustbox}
\usepackage{amsmath}

\usepackage{accents}

\usepackage{macros}

\usepackage[colorinlistoftodos]{todonotes}
\begin{document}

\begin{frontmatter}
\title{A structure-preserving numerical method for the compressible
Resistive-Hall-MHD system\tnoteref{t1}}
% \title{A structure-preserving numerical \mnblue{method} for \mnblue{the} compressible
% Resistive-Hall-MHD system}
%   \tnotetext[t1]{This research is funded by Swedish Research Council (VR)
%   under grant number 2021-04620.}
  \author[1]{Murtazo Nazarov}
  \ead{murtazo.nazarov@it.uu.se}
  \author[1]{Rafael Rodriguez-Velasco}
  \ead{rafael.rodriguez-velasco@it.uu.se}
  \author[2]{Ignacio Tomas}
  \ead{igtomas@ttu.edu}
  \address[1]{Department of Information Technology, Uppsala University, Sweden}
  \address[2]{Department of Mathematics \& Statistics Texas Tech University}
\tnotetext[t1]{Authors are listed alphabetically.}

\begin{abstract}
In this paper, we present a structure-preserving method for the compressible
resistive Hall-magnetohydrodynamics (MHD) model. The differential operator is
split into two parts: a hydrodynamic part consisting of the compressible Euler
equations, and a magnetic part consisting of a system coupling the Lorentz force
and the induction equation. The method uses continuous Lagrange elements for the
Euler part and a curl-conforming finite element space for the magnetic part. The
hydrodynamic part preserves the positivity of the density and internal energy,
the conservation of total energy, and the minimum principle for the specific
entropy. Owing to the choice of finite elements, the magnetic part preserves the
divergence involution constraint. The fluid part is solved using explicit
strong-stability-preserving Runge–Kutta (SSP-RK) methods, whereas the magnetic
part is solved by Crank-Nicholson method, which requires using Newton's method.
Coercivity estimates for the Jacobian of the corresponding Newton iteration are
presented. We introduce a high-order artificial resistivity to improve the
conditioning of the nonlinear residual and the invertibility of the Jacobian.
Several challenging benchmarks, including a smooth whistler wave, the
Orszag–Tang vortex for comparing resistive MHD with resistive Hall-MHD, and a
magnetic reconnection problem, are solved to validate the robustness and
accuracy of the method.

\end{abstract}


\end{frontmatter}

\section{Introduction}\label{sec:intro}

In this work we consider the numerical solution of the compressible Hall-MHD
system:
\begin{subequations}
\label{hallMHDsystem}
\begin{align}
\partial_t \rho
+ \diver{} \mom &= 0 \, ,\\
\partial_t \mom + \diver{}
(\rho^{-1}\mom \mom^\transp + \mathbb{I} p) &= \mu
\curl{}\Hfield \times \Hfield \\
\label{TotmePDE}
\partial_t \totme
+ \diver{}\big(\tfrac{\mom}{\rho} (\totme + p) \big) &=
\mu (\curl{}\Hfield \times \Hfield)\cdot \tfrac{\mom}{\rho}
+ \resist |\curl{}\Hfield|^2
\, , \\
\label{HeqnPDE}
\partial_t \Hfield - \curl{}(\tfrac{\mom}{\rho} \times \Hfield) &=
- \curl{}(\tfrac{\resist}{\mu} \curl{}\Hfield
+ \tfrac{d_i}{\rho} \curl{}\Hfield \times \Hfield)
\end{align}
\end{subequations}
where $\rho \in \mathbb{R}$ is the density, $\mom \in \mathbb{R}^3$ is the
momentum, $\mathbb{I} \in \mathbb{R}^{3 \times 3}$ is the identity matrix,
$\totme  \in \mathbb{R}$ is the total mechanical energy, and $\Hfield \in
\mathbb{R}^3$ is the magnetic field. Here the constant $\resist$ is the
resistivity while $d_i = \frac{m_i}{q_i}$: the constant $m_i > 0$ is the
specific ion mass while $q_i > 0$ is the specific ion charge, see for instance
\cite[p. 90]{Krall1974}. The term $\curl{}(\resist \curl{}\Hfield)$ in
the right hand side of \eqref{HeqnPDE} is resistive term, while the term
$\curl{}(\tfrac{d_i}{\rho} \curl{}\Hfield \times \Hfield)$ is the Hall term.
Note that resistive effects introduce the source of heat $\resist
|\curl{}\Hfield|^2$ in the right hand side of \eqref{TotmePDE}.

The MHD equations are widely used in astrophysics applications as well as
in nuclear fusion research \cite{Jardin2013, Freidberg1982}, where it is used
to study instabilities in plasma confinement \cite{Jorek2021, Sitenko1995}. The
MHD system can be understood as a formal asymptotic limit of the
two-fluid Euler-Maxwell model. More precisely it is the limit obtained under the
assumptions of infinite speed of light (equivalently, zero electric
permittivity, that is $\epsilon_0 \rightarrow 0^+$) and zero electron mass. Most
frequently, the Hall term is neglected during the derivation of the ideal MHD
model. The importance of the Hall term was first pointed out by James Lighthill
in \cite{Lighthill1960}. Derivations of the Hall-MHD system can be found in
\cite{Acherito2011, Wilmanski2015}. Since the early 2000s there has been a
growing body of scientific literature indicating that Resistive and Hall terms
are fundamental to reproduce magnetic reconnection rates \cite{Birn2001,
Shay2001, Prit2001, Homann2005} observed in practice.

Numerical solutions of the MHD system are vital to predict phenomena in various
scientific fields such as plasma physics and astrophysics. Furthermore, when
performing numerical simulations of the MHD system, it is crucial to ensure the
preservation of essential structure of the solution, such as positivity
properties, conservation of total energy, entropy-dissipation, and involution
constraints. For instance, the works of \cite{Balsara2012, Cheng2013, Wu2018,
Wu2018II}, along with references provided therein, represent
just a subset of the comprehensive research dedicated to achieving
positivity-preserving approximations for the compressible ideal MHD system.

Another very important property preserved by the ideal-MHD and
Resistive-Hall-MHD systems is the involution constraint of the magnetic field,
that is, the invariance of the magnetic field's divergence. Time-invariance of
the magnetic field implies that if the magnetic field is zero at initial time,
then, it should remain zero for all time. Without making any claim of
completeness, some references advancing numerical techniques that preserve some
aspect of the involution (e.g. locally divergence-free and divergence cleaning
methods) are \cite{Dedner2002, Li2005, Balsara2009, Londrillo2004}.

Overall, the literature on structure preserving methods for the ideal
compressible MHD system is reasonably developed. On the other hand, the list
mathematical references\footnote{By mathematical we mean: indexed by MathSciNet
database} advancing numerical schemes for the compressible Resistive-Hall-MHD
system is noticeably small. A fairly complete list of references obtained from
MathSciNet database, with specific focus on compressible Resistive-Hall-MHD, is
\cite{Toth2008, Toth2014, Balbas2009, Derigs2018, Strumik2017, Arnold2008,
Chacon2025}. To the best of our knowledge, there is no literature with a focus
on structure preservation of the compressible Resistive-Hall-MHD system.

The development of numerical methods for the MHD systems is inextricably
associated to the divergence formulation. The divergence formulation of the
ideal MHD system offers a few challenges that are hard to ignore. Among them,
the divergence formulation is not Galilean invariant, it is not symmetrizable,
and its Jacobian is defective (it does not possess a full set of eigenvectors).
From the purely computational perspective, the Riemann problem of the ideal MHD
in divergence form is not well defined unless the condition
$[\![\Bfield]\!]\cdot\normal = 0$ holds across discontinuity surfaces
\cite{Torr2003}. Computationally, some of these problems may be alleviated with
the inclusion of modifications of the scheme, for instance, with the use of
constraint transport techniques, see for instance \cite{Ross2006} and
references therein.

Regarding the existence of solutions for compressible Hall-MHD models very
little is known. An extensive literature search reveals that most, if not all,
the Analysis literature is limited to short-time existence of strong
solutions \cite{Fan2015,Fan2016,Gao2016, Tao2017, Lai2019}. Notably, to the best
of our knowledge, there are no existence results for the compressible Hall-MHD
model without resistivity. On the other hand, the role of resistivity, for the
incompressible Hall-MHD model, is well understood: resistivity is fundamental to
achieve well-posedness of the incompressible Hall-MHD model. In a series of highly
cited publications \cite{Degond2014, Chae2016, Jeong2022} it was
proven that the incompressible Hall-MHD model without resistivity is ill-posed.
Based on this record of Analysis results, it may be reasonable to \emph{assume}
that resistivity is fundamental to guarantee well-posedness of the compressible
Hall-MHD model as well.

The current work is a continuation of the ideas advanced in \cite{Dao2024}. In
that work, the authors advanced a proof of the minimum principle of the specific
entropy as well as entropy-dissipation inequalities that do \emph{not} involve
viscous regularization of the magnetic field. These results indicate that the
ideal MHD is not a conservation law in the vanishing-viscosity sense of
\cite{Lax1957,Chueh1977, Bianchini2005}. Inspired by this result, we developed a
numerical scheme that decomposes the ideal MHD system into Euler's equation and a 
purely Hamiltonian PDE that couples the Lorentz force and the induction
equation. Such scheme is capable of preserving positivity properties, total
energy, involution constraints, and entropy-dissipation inequalities with no
divergence-cleaning or related tools. Most importantly, the induction equation
is not stabilized in any form or fashion: which is entirely consistent with the
vanishing-viscosity argument.

In this paper, we extend the ideas of \cite{Dao2024} to the case of the
Resistive-Hall-MHD model. We develop a new scheme that preserves positivity of
the density, positivity of the internal energy, total mechanical energy,
minimum principle of the specific entropy. We also prove that the scheme
preserves entropy-dissipation inequalities if the numerical method used to
solve Euler's system preserves such a property. The scheme is semi-implicit:
Euler's system is advanced explicitly, while the source system consisting of
the Lorentz force coupled to the induction equation is advanced in a
time-implicit fashion.

One of the most important differences between the present work and
\cite{Dao2024} is the behaviour of nonlinear solvers. In our previous work we
used Newton's method to solve the nonlinear residual associated to the coupled
system involving Lorentz force and the induction equation. In that work,
nonlinear performance turned out to be exceptional across several tests
and mesh types: Newton's solver never exceeded 4 iterations per
time-step. The addition of the Resistive term is benign and does not change
the behaviour/performance of Newton's scheme. However, the addition of the Hall
terms is rather delicate. Our initial computations for the Resistive-Hall-MHD
system, using uniform structured meshes, delivered comparable performance along
the lines of at most 4-5 Newton iterations per time step. However, such results
did not translate to other mesh types. Nonlinear solver performance turned out
to be rather inconsistent across several types of meshes: either too many Newton
iterations were required or the Jacobian was severely ill-conditioned. This is
not strictly speaking a defect of the scheme, but rather a problem associated
with our choice of nonlinear solver. Since we were unwilling to give up Newton's
method\footnote{For instance, we could have replaced Newton's method with some
form of accelerated fixed-point method.} and its second-order convergence,
we decided to modify the scheme. Coercivity analysis of the Jacobian
reveals a loss of invertibility whenever the electron speed is too large.
Therefore, we devised an artificial resistivity for the induction equation that
improves well-posedness of the Jacobian, thereby stabilizing the behaviour of
nonlinear Newton iterations.

One of the most important aspects of scientific computing is Verification and
Validation \cite{Babus2004}. Generally speaking, verification consists in
checking that the method is convergent. The standard quantitative test is
computation of convergence rates (or log-log plots) using exact
solutions\footnote{An alternative approach is using the method of manufactured
solutions. However, such an approach does not evaluate the ability of the scheme
to approximate autonomous dynamics. Overall, in the context of
nonlinear hyperbolic-like problems, manufactured solutions is not a highly
regarded approach for code verification.}. However, exact analytical solutions
for the compressible-Resistive-MHD model are not available at this point in
time. This makes quantitative evaluation of the scheme very difficult. On the
other hand, qualitative tests are just limited to graphical comparison. Again,
given the extremely rarefied body of literature on this model, there are just a
few meaningful computational results, most of which use very coarse meshes to
be considered reference results. For instance, with the exception of
\cite{Chacon2025}, pretty much every computation of the Resistive-Hall-MHD
model available in the literature uses 128$\times$128 cells for the GEM
challenge problem.

In this regard, one of our most important contributions is the development of
both quantitative tests and reference computational results. We present a
quantitative test consisting of a smooth-traveling wave, in the linear
resistive regime that allows us to verify some aspects of the accuracy of the
method. We illustrate the computation of the GEM challenge with resolutions of
up to 1024$\times$1024 elements. Our long-term goal is to make the
corresponding data available to the wider scientific community. We also advance
direct comparisons of Resistive-MHD vs Resistive-Hall-MHD with Orzag-Tang
vortex test using resolutions of up to 724$\times$724 cells. Finally, to the
best of our knowledge, there are no computations of the Resistive-Hall-MHD
model using unstructured meshes, with most, if not all results using structured
cartesian meshes. It is well-known that most numerical methods use to solve
hyperbolic-like problems are very sensitive to mesh-imprint artifacts. In this
work, we present a series of results using structured isotropic meshes,
structured anisotropic, fully unstructured quasi-uniform meshes, and criss-cross
meshes, showing that our method delivers comparable results regardless of the
choice of mesh.

% \RED{***The following paragraph in Red has to be revisited in its entirety once
% the paper is written***} \\

{The paper is organized as follows: in Section \ref{sec:split} we present a
splitting of the differential operator \eqref{hallMHDsystem}, prove the
properties preserved by each operator and show that the splitting is compatible
with the sequential preservation of such properties. In Section
\ref{sec:SpaceDisc} we summarize the notation related to the space
discretization. In Section \ref{Sec:NumericalSource} we provide the space and
time discretization of the source-system associated to the Lorentz force
and the induction equation (containing the ideal, Hall, and resistive terms).
The main theoretical results of this scheme are independent of the choice of
numerical scheme used to solve Euler's system. Therefore, in Section
\ref{sec:hypsolver_ass} we outline the main assumptions made about the
hyperbolic solver. The precise choice of hyperbolic solver used for all our
computations is described in Appendix B of our previous work \cite{Dao2024}. In
Section \ref{Sec:ArtResMotiv} we elaborate our motivations for the development
of an artificial viscosity. In Section \ref{Sec:art_resist} we make precise the
artificial viscosity used for all our computations. In Section
\ref{sec:all_the_scheme} we provide an algorithmic summary of
the whole operator splitting scheme. In Section \ref{sec:numeric} we
present our numerical results.}

{Finally, we mention that this paper has three appendices. \ref{app:thermo}
contains a summary of thermodynamic properties which are used in the context of
Sections \ref{hallMHDsystem} and \ref{Sec:NumericalSource}. \ref{app:jacobian}
contains the derivation of the Jacobian and its coercivity analysis: this
appendix and its contents play a very important role in this paper.
\ref{app:two_and_half} describes the implementation of the 2.5-space dimensions
implementation.}

% In Section \ref{sec:hypsolver_ass} we summarize the assumptions made about
% the
% hyperbolic solver used to solve Euler's system. In this work the choice of
% hyperbolic solver is

%%%%%%%%%%%%%%%%%%%%%%%%%%%%%%%%%%%%%%%%%%%%%%%%%%%%%%%%%%%%%%%%%%%%%%%%%%%%%%%%
%%%%%%%%%%%%%%%%%%%%%%%%%%%%%%%%%%%%%%%%%%%%%%%%%%%%%%%%%%%%%%%%%%%%%%%%%%%%%%%%
%%%%%%%%%%%%%%%%%%%%%%%%%%%%%%%%%%%%%%%%%%%%%%%%%%%%%%%%%%%%%%%%%%%%%%%%%%%%%%%%

\section{Splitting of the differential operator}\label{sec:split}

We start by splitting \eqref{hallMHDsystem} into two differential operators
\begin{align}
\label{OperatorOne}
\text{Operator }\#1
\left\{
\begin{aligned}
\partial_t \rho
+ \diver{} \mom &= 0 \, ,\\
\partial_t \mom + \diver{}
(\rho^{-1}\mom \mom^\transp + \mathbb{I} p) &= \bzero \\
%
% \label{TotmePDE}
\partial_t \totme
+ \diver{}\big(\tfrac{\mom}{\rho} (\totme + p) \big) &= 0 \, , \\
\partial_t \Hfield &= \bzero
\end{aligned}
\right.
\end{align}
and
\begin{align}
\label{OperatorTwo}
\text{Operator }\#2
\left\{
\begin{aligned}
\partial_t \rho &= 0 \, ,\\
\partial_t \mom - \mu \curl{}\Hfield \times \Hfield &= \bzero \\
\partial_t \totme - \mu (\curl{}\Hfield \times \Hfield)\cdot \tfrac{\mom}{\rho}
- \resist |\curl{}\Hfield|^2 &= 0\\
\partial_t \Hfield - \curl{}(\tfrac{\mom}{\rho} \times \Hfield)
+ \curl{}(\tfrac{\resist}{\mu} \curl{}\Hfield
+ \tfrac{d_i}{\rho} \curl{}\Hfield \times \Hfield)
&= \bzero
\end{aligned}
\right.
\end{align}
We briefly discuss the properties preserved by each operator. We will assume
that the pressure is computed from a complete Equation of State (EOS) as
described by \eqref{PressureEpist}. In particular, we assume that the EOS
satisfies the assumptions \eqref{posTempAssump} and \eqref{convexAssump}, see
\ref{app:thermo} for more details. In Section \ref{Sec:NumericalScheme} we
outline the development of numerical schemes that preserve such properties in
the fully-discrete setting.

\begin{proposition}[Properties preserved by Operator \#1]
\label{prop:op_one_prop} Let $t_1 <
t_2$, consider the time interval $[t_1, t_2]$ and that $\mom\cdot\normal
= 0$ in the entirety of the boundary $\partial\domain$, compactly supported
initial data, and that $t_2 - t_1$ is small enough such that no wave
reaches the boundary for any $t \in [t_1, t_2]$. Alternatively, assume periodic
boundary conditions. Then, the solution operator \eqref{OperatorOne} satisfies
the following properties:
\begin{align}
\label{ConsPropGlobal}
\begin{aligned}
\int_{\Omega} \rho(\xcoord, t_2) \dx = \int_{\Omega} \rho(\xcoord, t_1)\dx \ ,
\ \
\int_{\Omega} \mom(\xcoord, t_2) \dx = \int_{\Omega} \mom(\xcoord, t_1)\dx \\
\int_{\Omega} \totme(\xcoord, t_2) \dx = \int_{\Omega} \totme(\xcoord, t_1)\dx
\ ,  \ \
\int_{\Omega} \Hfield(\xcoord, t_2) \dx = \int_{\Omega} \Hfield(\xcoord,
t_1)\dx
\end{aligned}
\end{align}
% \RED{*** the statement of the momentum is wrong ****} \\
Note that the conservation properties on the magnetic field follow trivially
since $\Hfield(\xcoord, t_2)$ $\equiv \Hfield(\xcoord, t_1)$ in the context of
Operator \#1. Therefore, we also have that
\begin{align}
\label{OpOneTotalNRG}
\int_{\Omega} \totme(\xcoord, t_2) + \tfrac{\mu}{2} |\Hfield(\xcoord, t_2)|^2\dx
=
\int_{\Omega} \totme(\xcoord, t_1) + \tfrac{\mu}{2} |\Hfield(\xcoord, t_1)|^2\dx
\end{align}

Regarding pointwise properties: we have that
\begin{align}
\label{EulerPositDens}
\inf_{(\xcoord,t) \in \Omega \times [t_1,t_2]} \rho(\xcoord,t) \geq 0 \, ,
\end{align}
provided that the initial data $\rho(\xcoord, t_1) \geq 0$ for all $\xcoord \in
\Omega$, for the specific entropy we have that:
\begin{align*}
\inf_{\xcoord \in \Omega }
s(\rho(\xcoord, t_2),e(\state(\xcoord, t_2))
\geq
\inf_{\xcoord \in \Omega} s(\rho(\xcoord, t_1),e(\state(\xcoord, t_1))
\end{align*}
and the mathematical entropy satisfies
\begin{align}
\label{EntDissipPointwise}
\partial_t \eta(\state) + \diver{} \mathbb{q}(\state) \leq 0 \ \
\text{for all } (\xcoord, t) \in \domain \times [t_1, t_2]
\end{align}
where $\{\eta, \mathbb{q}\} = \{- \rho s, -\mom s\}$ is the entropy-flux pair.
Integration on $\Omega$ of the pointwise estimate \eqref{EntDissipPointwise}
naturally leads to the inequality:
\begin{align*}
\int_{\domain} \eta(\state(\xcoord, t_2)) \dx
\leq
\int_{\domain} \eta(\state(\xcoord, t_1)) \dx
\end{align*}
\end{proposition}

We do not advance a proof of this proposition since it is a recollection of
well-known results about Euler's system. Conservation properties
\eqref{ConsPropGlobal} are just a consequence of the divergence theorem. The
non-negativity of the density \eqref{EulerPositDens} is proved, for instance, in
\cite{Guermond2014}. For the specific case of the ideal gas equation of state,
the minimum principle of the specific entropy was proved for the first time in
\cite{Tadmor1986}. For the case of (arbitrary) thermodynamically
stable\footnote{An equation of state is thermodynamically stable if the
specific internal energy $e = e(s,v)$ is a convex function with respect to the
specific entropy $s$ and specific volume $v$. See \ref{app:thermo} for more
details.} equations of state, a proof
of the minimum principle of the specific entropy can be found in
\cite{Guermond2014}. The entropy-dissipation inequality
\eqref{EntDissipPointwise} is a consequence of the vanishing viscosity principle
and the convexity of $\eta(\state)$ with respect to $\state = [\rho, \mom,
\totme]^\transp$, see for instance \cite{Godlewski1996}. \\

The following proposition makes significant use the thermodynamics summary in
\ref{app:thermo}. The reader unfamiliar with the basics of thermodynamics is
encouraged to read the Appendix before reading Proposition
\ref{prop:op_one_prop}.

\begin{proposition}[Properties preserved by Operator \#2] Let $t_1 < t_2$,
assume that the initial data is such that
\begin{align}
\label{OpTwoPointAss}
\begin{gathered}
\rho(\xcoord, t_1) > 0 \  , \ \
e(\xcoord, t_1) := (\totme - \tfrac{1}{2}\rho |\vel|^2)(\xcoord, t_1) > 0 \\
\ \ \text{and} \ \
\theta(\rho(\xcoord, t_1), e(\xcoord, t_1)) := \big[\tfrac{\partial}{\partial e}
s(\tfrac{1}{\rho(\xcoord, t_1)}, e(\xcoord, t_1))\big]^{-1} > 0
\ \ \text{for all }\xcoord \in \domain ,
\end{gathered}
\end{align}
here $\theta(\rho(\xcoord, t_1), e(\xcoord, t_1))$ is the temperature at time
$t_1$. Then, Operator \#2 as described in \eqref{OperatorTwo}, preserves the
following properties:
\begin{align}
& \begin{aligned}\label{OpTwoPropOne}
& \int_{\domain} (\tfrac{1}{2}\rho|\vel|^2
+ \tfrac{\mu}{2}|\Hfield|^2)(\xcoord, t_2) \dx \\
& \ \ \ + \int_{t_1}^{t_2} \int_{\partial\domain} \Big(
\resist \curl{}\Hfield - \mu (\vel \times \Hfield) + \tfrac{\mu d_i}{\rho}
\curl{}\Hfield \times
\Hfield \Big)
\cdot (\Hfield \times \normal) \ds \mathrm{d}t \\
& \ \ \ =
\int_{\domain} (\tfrac{1}{2}\rho|\vel|^2
+ \tfrac{\mu}{2}|\Hfield|^2)(\xcoord, t_1) \dx
- \int_{t_1}^{t_2} \int_{\domain} \resist |\curl{}\Hfield|^2 \dx  \mathrm{d}t
\end{aligned} \\
\label{OpTwoInte}
&(\totme - \tfrac{1}{2} \rho|\vel|^2)(\xcoord, t_2) =
(\totme - \tfrac{1}{2} \rho|\vel|^2)(\xcoord, t_1)
+ \int_{t_1}^{t_2} \resist |\curl{}\current|^2 \mathrm{d}t \\
\label{OpTwoTemp}
&\theta(\rho(\xcoord, t_2), e(\xcoord, t_2))
\geq
\theta(\rho(\xcoord, t_1), e(\xcoord, t_1)) \\
\label{OpTwoLittleS}
&s(\rho(\xcoord, t_2), e(\xcoord, t_2))
\geq
s(\rho(\xcoord, t_1), e(\xcoord, t_1)) \\
\label{OpTwoEta}
&\eta(\rho(\xcoord, t_2), e(\xcoord, t_2))
\leq
\eta(\rho(\xcoord, t_1), e(\xcoord, t_1))
\end{align}
Note that properties \eqref{OpTwoInte}-\eqref{OpTwoEta} are pointwise
properties, that is, they hold for every $\xcoord \in \domain$.
\end{proposition}

%%%%%%%%%%%%%%%%%%%%%%%%%%%%%%%%%%%%%%%%%%%%%%%%%%%%%%%%%%%%%%%%%%%%%%%%%%%%%%%%
%%%%%%%%%%%%%%%%%%%%%%%%%%%%%%%%%%%%%%%%%%%%%%%%%%%%%%%%%%%%%%%%%%%%%%%%%%%%%%%%
%%%%%%%%%%%%%%%%%%%%%%%%%%%%%%%%%%%%%%%%%%%%%%%%%%%%%%%%%%%%%%%%%%%%%%%%%%%%%%%%

% \todo[inline]{Why this proof has a number 2.1?}
% Answer: it's not our problem. It's a problem of the editors. They will
% take care of it.

\begin{proof} We start by noting that since $\partial_t \rho = 0$, Operator \#2
can be rewritten as:
\begin{subequations}
\begin{align}
\label{OperatorTwoVel}
\rho \partial_t \vel - \mu \curl{}\Hfield \times \Hfield &= \bzero \\
\label{OperatorTwoTotme}
\partial_t \totme - \mu (\curl{}\Hfield \times \Hfield)\cdot \vel
- \resist |\curl{}\Hfield|^2 &= 0\\
\label{OperatorTwoH}
\partial_t \Hfield + \curl{}(\tfrac{\resist}{\mu} \curl{}\Hfield
- \vel \times \Hfield
+ \tfrac{d_i}{\rho} \curl{}\Hfield \times \Hfield)
&= \bzero
\end{align}
\end{subequations}
Now, we multiply \eqref{OperatorTwoVel} by $\vel$ and \eqref{OperatorTwoH} by
$\mu \Hfield$ and integrate in space to obtain:
\begin{align}
\label{SourceMomEnergy}
\int_{\domain} \partial_t \big(\tfrac{1}{2}\rho |\vel|^2\big)
- \mu (\curl{}\Hfield \times \Hfield)\cdot \vel \dx &= \bzero \\
\nonumber
\int_{\domain} \tfrac{1}{2} \mu |\Hfield|^2
+ \curl{}(\resist \curl{}\Hfield - \mu \vel \times \Hfield
+ \tfrac{\mu d_i}{\rho} \curl{}\Hfield \times \Hfield)\cdot\Hfield \dx
&= \bzero \, .
\end{align}
Now, integrating by parts this last equation we obtain:
\begin{align}
\label{sourceIntbyParts}
\begin{aligned}
&\int_{\domain} \tfrac{1}{2} \mu |\Hfield|^2
+ (\resist \curl{}\Hfield - \mu \vel \times \Hfield
+ \tfrac{\mu d_i}{\rho} \curl{}\Hfield \times \Hfield)\cdot\curl{}\Hfield \dx \\
&\ \ \  + \int_{\partial\domain} (\resist \curl{}\Hfield - \mu \vel \times
\Hfield
+ \tfrac{\mu d_i}{\rho} \curl{}\Hfield \times \Hfield) \cdot (\Hfield \times
\normal) \ds
= \bzero \, ,
\end{aligned}
\end{align}
adding \eqref{SourceMomEnergy} to \eqref{sourceIntbyParts} we obtain:
\begin{align}
\begin{aligned}
& \tfrac{\partial}{\partial t} \int_{\domain} \tfrac{1}{2}\rho|\vel|^2
+ \tfrac{\mu}{2}|\Hfield|^2 \dx \\
& \ \ + \int_{\partial\domain} \Big(
\resist \curl{}\Hfield - \mu (\vel \times \Hfield) + \tfrac{\mu d_i}{\rho}
\curl{}\Hfield \times
\Hfield \Big)
\cdot (\Hfield \times \normal) \ds
=
- \int_{\domain} \resist |\curl{}\Hfield|^2 \dx \\
\end{aligned}
\end{align}
Integrating this expression in time between time $t_1$ and $t_2$ then
\eqref{OpTwoPropOne} follows. Now we multiply \eqref{OperatorTwoVel} by $\vel$
and subtract the result from \eqref{OperatorTwoTotme} to obtain
\begin{align}
\label{Op2consEnergyII}
\tfrac{\partial}{\partial t} \big(\totme - \tfrac{1}{2} \rho|\vel|^2 \big)
= \resist |\curl{}\current|^2 .
\end{align}
Integrating \eqref{Op2consEnergyII} in time between time $t_1$ and $t_2$ then
\eqref{OpTwoInte} follows. Since $\partial_t \rho = 0$ in the context of
Operator \#2, we can divide \eqref{Op2consEnergyII} by $\rho$ to obtain:
\begin{align}
\label{Op2consEnergyIII}
\tfrac{\partial e}{\partial t}
=
\tfrac{\partial}{\partial t}
\big(\tfrac{\totme}{\rho} - \tfrac{1}{2}|\vel|^2 \big)
=
\tfrac{\resist}{\rho} |\curl{}\current|^2 \, ,
\end{align}
therefore the specific internal energy can only increase during the evolution
of Operator \#2. Since $\partial_t \rho = 0$ and $\frac{\partial
\theta}{\partial e} \geq 0$, see convexity Assumption
\eqref{convexAssump}-\eqref{tempMono}, then \eqref{Op2consEnergyIII} implies
that the temperature $\theta$ can only increase during the evolution of Operator
\#2. More precisely, using the fundamental theorem of calculus, the chain rule,
the thermodynamic relationships in \eqref{PressureEpist}, and identity
\eqref{Op2consEnergyIII} we have that:
\begin{align}
\label{PositThetaProof}
\begin{aligned}
&\theta(\rho(\xcoord, t_2), e(\xcoord, t_2))
- \theta(\rho(\xcoord, t_1), e(\xcoord, t_1)) \\
& \ \ \ = \int_{t_1}^{t_2} \partial_t \theta(\rho,e) \mathrm{d}t
= \int_{t_1}^{t_2}
\frac{\partial \theta}{\partial e} \frac{\partial e}{\partial t}\mathrm{d}t \\
& \ \ \ =
\int_{t_1}^{t_2}
\frac{\partial }{\partial e} \Big[\frac{\partial s}{\partial e}\Big]^{-1}
\frac{\partial e}{\partial t}\mathrm{d}t
=
- \int_{t_1}^{t_2}
\Big[\frac{\partial s}{\partial e}\Big]^{-2}
\frac{\partial^2 s}{\partial^2 e}
\frac{\partial e}{\partial t}\mathrm{d}t \\
&\ \ \ =
- \int_{t_1}^{t_2}
\theta^2 \frac{\partial^2 s}{\partial^2 e}
\frac{\resist}{\rho} |\curl{}\current|^2 \mathrm{d}t \geq 0 \, \ \  \
\text{for all }\xcoord \in \domain
\end{aligned}
\end{align}
where the last inequality follows from the convexity assumption
$\frac{\partial^2 s}{\partial^2 e} \leq 0 $, see \eqref{tempMono}. Combining
assumption \eqref{OpTwoPointAss} with \eqref{PositThetaProof} establishes
that $\theta(\rho(\xcoord, t), e(\state(\xcoord, t)))$ can only take
strictly positive values in the interval $[t_1, t_2]$. Similarly, since
$\partial_t \rho = 0$ in the context of Operator \#2, using the thermodynamic
relationship $\tfrac{\partial
s}{\partial e} = \frac{1}{\theta}$, and identity \eqref{Op2consEnergyIII} we
obtain:
\begin{align}
\label{sminProof}
\begin{aligned}
& s(\rho(\xcoord, t_2), e(\state(\xcoord, t_2)))
- s(\rho(\xcoord, t_1), e(\state(\xcoord, t_1))) \\
& \ \ \ = \int_{t_1}^{t_2} \partial_t s(\rho,e) \mathrm{d}t
= \int_{t_1}^{t_2} \frac{\partial s}{\partial e}
 \frac{\partial e}{\partial t} \mathrm{d}t
= \int_{t_1}^{t_2} \frac{1}{\theta}
\frac{\resist}{\rho} |\curl{}\current|^2 \mathrm{d}t \geq 0
\end{aligned}
\end{align}
where the inequality follows from the fact that $\theta$ can only take positive
values in the interval $[t_1, t_2]$. Therefore \eqref{sminProof} leads to the
proof of \eqref{OpTwoLittleS}. Now, reorganizing \eqref{sminProof} and
multiplying both sides of the equality by $ - \rho(\xcoord, t_1)$ we obtain:
\begin{align}
\label{etaDissPointI}
\begin{aligned}
& - \rho(\xcoord, t_1) s(\rho(\xcoord, t_2), e(\state(\xcoord, t_2))) \\
& \ \ \ = - \rho(\xcoord, t_1) s(\rho(\xcoord, t_1), e(\state(\xcoord, t_1)))
- \rho(\xcoord, t_1) \int_{t_1}^{t_2} \frac{1}{\theta}
\frac{\resist}{\rho} |\curl{}\current|^2 \mathrm{d}t
\end{aligned}
\end{align}
Since $\eta = -\rho s$ and $\rho(\xcoord, t_1) = \rho(\xcoord, t_2)$, then
\eqref{etaDissPointI} can be rewritten in a more compact form as:
\begin{align}
\label{etaDissPointII}
\begin{aligned}
& \eta_2 = \eta_1 - \rho_1 \int_{t_1}^{t_2} \frac{1}{\theta}
\frac{\resist}{\rho} |\curl{}\current|^2 \mathrm{d}t
\end{aligned}
\end{align}
which proves \eqref{OpTwoEta}.
\end{proof}

\begin{corollary}[Total energy balance]\label{corollaryTotalE} We note that
global balance
\eqref{OpTwoPropOne} and pointwise property \eqref{OpTwoInte} imply a global
balance of total energy. More precisely, integrating \eqref{OpTwoInte} in space
we obtain:
\begin{align}
\label{OpTwoInteInt}
\int_{\domain} (\totme - \tfrac{1}{2} \rho|\vel|^2)(\xcoord, t_2) \dx
=
\int_{\domain} (\totme - \tfrac{1}{2} \rho|\vel|^2)(\xcoord, t_1) \dx
+ \int_{\domain}  \int_{t_1}^{t_2} \resist |\curl{}\current|^2 \mathrm{d}t \dx
\end{align}
Now, adding \eqref{OpTwoInteInt} to \eqref{OpTwoPropOne} yields:
\begin{align*}
& \int_{\domain} (\totme
+ \tfrac{\mu}{2}|\Hfield|^2)(\xcoord, t_2) \dx \\
& \ \ \ + \int_{t_1}^{t_2} \int_{\partial\domain}
\Big(\resist \curl{}\Hfield - \mu (\vel \times \Hfield) + \tfrac{\mu d_i}{\rho}
\curl{}\Hfield \times
\Hfield \Big)
\cdot (\Hfield \times \normal) \ds \mathrm{d}t \\
& \ \ \ = \int_{\domain} (\totme + \tfrac{\mu}{2}|\Hfield|^2)(\xcoord, t_1) \dx
\end{align*}
which is the balance of total energy. In particular, if we consider
periodic boundary conditions, or $\Hfield \times \normal \equiv \bzero$ on the
entirety of the boundary, it becomes clear that Operator \#2 will preserve
total energy.
\end{corollary}

% \newpage

%%%%%%%%%%%%%%%%%%%%%%%%%%%%%%%%%%%%%%%%%%%%%%%%%%%%%%%%%%%%%%%%%%%%%%%%%%%%%%%%
%%%%%%%%%%%%%%%%%%%%%%%%%%%%%%%%%%%%%%%%%%%%%%%%%%%%%%%%%%%%%%%%%%%%%%%%%%%%%%%%
%%%%%%%%%%%%%%%%%%%%%%%%%%%%%%%%%%%%%%%%%%%%%%%%%%%%%%%%%%%%%%%%%%%%%%%%%%%%%%%%

% \newpage
\section{Numerical scheme}\label{Sec:NumericalScheme}

\subsection{Space discretization preliminaries}\label{sec:SpaceDisc}

In this subsection we outline the space discretization used for Euler's
components $\{\rho, \mom, \totme\}$ and the magnetic field $\Hfield$. Let
$\domain \subset \mathbb{R}^d$, with $d = 2$ or $d=3$, we consider a
simplicial mesh $\triangulation$ and a corresponding scalar-valued continuous
finite element space $\FESpaceHypComp$ for each
component of Euler's system:
\begin{align}\label{VdefSpace}
\FESpaceHypComp &=
\big \{ v_h(\xcoord) \in \mathcal{C}^{0}(\Omega) \;\big|\;
v_h (\locglobmap_\element(\widehat{\xcoord})) \in
\mathbb{P}^1(\widehat{\element}) \;\forall \element \in \triangulation \big\}.
\end{align}
Here, $\locglobmap_\element(\widehat{\xcoord}):\widehat{\element}\to\element$
denotes a diffeomorphism mapping from the unit simplex $\widehat{\element}$ to
the physical element $\element \in \triangulation$,
and $\mathbb{P}^1(\widehat{K})$ is polynomial space of at most first degree on
the reference element. We
define $\HypVertices=\big\{1:\text{dim}(\FESpaceHypComp)\big\}$ as the
index-set of global, scalar-valued degrees of freedom (DOF) corresponding to
$\FESpaceHypComp$. Similarly, we introduce the set of global shape functions
$\{\phi_{i}(\xcoord)\}_{i \in \HypVertices}$ and the set of collocation
points $\{\xcoord_{i}\}_{i \in \HypVertices}$ satisfying the property
$\phi_{i}(\xcoord_j) = \delta_{ij}$ for all $i,j \in \HypVertices$. We
assume that the partition of unity property $\sum_{i \in \HypVertices}
\phi_{i}(\xcoord) = 1$ holds true for all $\xcoord \in \Omega$. We
introduce a number of matrices that will be used for the algebraic
discretization. We define the consistent mass matrix entries $m_{ij} \in
\mathbb{R}$, lumped mass matrix $m_i \in \mathbb{R}$, and the discrete
divergence-matrix entries $\bv{c}_{ij} \in \mathbb{R}^d$:
\begin{align}
  \label{schemeMatrices}
  m_{ij} = \int_{\Omega} \HypBasisComp_i \HypBasisComp_j\dx \ , \ \
  m_i = \int_{\Omega} \HypBasisComp_i \dx \ , \ \
  \bv{c}_{ij} = \int_{\Omega} \nabla\phi_j \phi_i \dx \, .
\end{align}
Note that the definition of $m_{ij}$ and the partition of unity property
$\sum_{i \in \HypVertices} \phi_{i}(\xcoord) = 1$ imply that $\sum_{j \in
  \HypVertices} m_{ij} = m_i$. Given two scalar-valued finite element
functions $u_h = \sum_{i \in \HypVertices} u_i \HypBasisComp_i \in
\FESpaceHypComp$ and $v_h = \sum_{i \in \HypVertices} v_i
\HypBasisComp_i\in \FESpaceHypComp$ we define the lumped inner product as:
\begin{align}\label{LumpedInner}
  \langle u_h, v_h \rangle = \sum_{i \in \HypVertices} m_i u_i v_i \, .
\end{align}
Similarly, we define the space of vector-valued finite element functions
$[\FESpaceHypComp]^d$. For the case of vector-valued functions
$\boldsymbol{u}_h = \sum_{i \in
\HypVertices} \boldsymbol{u}_i \HypBasisComp_i \in
[\FESpaceHypComp]^d$ and $\boldsymbol{v}_h = \sum_{i \in \HypVertices}
\boldsymbol{v}_i \HypBasisComp_i\in [\FESpaceHypComp]^d$ the lumped
inner-product is defined as $\langle \boldsymbol{u}_h, \boldsymbol{v}_h \rangle
= \sum_{i \in \HypVertices} m_i \boldsymbol{u}_i \cdot \boldsymbol{v}_i$.

We define the curl-conforming finite-dimensional space:
\begin{align}\label{BDMspace}
\FESpaceH = \big\{ \Htest_h \in H(\text{curl}, \Omega) \, \big | \,
[\nabla_{\widehat{\xcoord}}\locglobmap_{\element}(\widehat\xcoord)]^{\transp}
\Htest_h(\locglobmap_{\element}(\widehat{\xcoord})) \in
[\mathbb{P}^1(\widehat{\element})]^d
\  \forall \element \in \triangulation \big\}
\end{align}
which will be used to discretize the magnetic field $\Hfield$. The finite
element space $\FESpaceH$ is known as the ``rotated'' or curl-conforming
$\text{BDM}_1$ space. The primary motivation to use this space is that it is the
simplest curl-conforming finite element that spans all the vector-valued
polynomial space $[\mathbb{P}_1]^d$, therefore full second-order
accuracy should be expected in the $L^p$-norm, for $1 \leq
p \leq \infty$, when using this element. We note in passing, that definition
\eqref{BDMspace} involves the covariant Piola transform, see \cite[Chapter
9]{ErnGuermond2021}.

Finally, we define the space
\begin{align}\label{PotSpace}
  \FESpacePot = \big\{ \EpotTest_h \in \mathcal{C}^0(\Omega) \, \big | \,
  \EpotTest_h(\locglobmap_{\element}(\widehat{\xcoord})) \in
  \mathbb{P}_2(\widehat{\element})
  \  \forall \element \in \triangulation \big\} \, .
\end{align}
It is easy to prove that the space $\FESpacePot$ satisfies the inclusion
$\nabla\FESpacePot \subset \FESpaceH$, more precisely, these two
spaces are part of a discrete exact sequence, see \cite{Arnold2014}. The space
$\FESpacePot$ is used to define the weak divergence-free property, see
Proposition \ref{Prop:sourceupdate}. For the specific case that we may want to
enforce tangential boundary conditions on the solution we also define the
finite element spaces
\begin{align*}
\FESpaceHtangent &= \big\{ \Htest_h \in \FESpaceH \, \big| \, \Htest_h \times
\normal = \bzero \text{ on } \partial\domain \big\} , \\
\FESpacePotZero &= \big\{ \EpotTest_h \in \FESpacePot \, \big| \, \EpotTest_h =
0 \text{ on } \partial\domain \big\} .
\end{align*}
Note again that the inclusion $\nabla \FESpacePotZero \subset \FESpaceHtangent$
holds true. Finally, we note that it is possible to develop compatible
construction of finite element spaces $\FESpaceHypComp$, $\FESpaceH$ and
$\FESpacePot$ of higher polynomial degree for both simplices as well as
quadrilateral/hexahedral elements. The reader is referred to Remark 3.1 in
\cite{Dao2024}.

% Finally we define:
% \begin{align}
% \label{LtwoProjector}
% \Pi_{\FESpaceHypComp}: [L^2(\domain)]^d \rightarrow [\FESpaceHypComp]^d
% \ \ \ \RED{\text{***Erase?***}}
% \end{align}
% the $L^2(\domain)$ projection onto the continuous finite element
% space $[\FESpaceHypComp]^d$. Note that computing the action of
% $\Pi_{\FESpaceHypComp}$ requires a mass-matrix inversion.
% $\Pi_{\FESpaceHypComp}$ and
Given an arbitrary vector-valued function $\boldsymbol{w}(\xcoord) \in
[L^2(\domain)]^d$ we define the lumped projection
$\Pi_{\FESpaceHypComp}^{\mathrm{L}}: [L^2(\domain)]^d \rightarrow
[\FESpaceHypComp]^d$ as:
\begin{align}
\label{LumpedL2}
\Pi_{\FESpaceHypComp}^{\mathrm{L}} [\boldsymbol{w}(\xcoord)]
:=
\sum_{i \in \HypVertices} \boldsymbol{w}_i \HypBasisComp_i(\xcoord)
\ \ \text{where} \ \
\boldsymbol{w}_i := \frac{1}{m_i} \int_{\domain} \boldsymbol{w}(\xcoord)
\HypBasisComp_i(\xcoord) \dx
\end{align}
We will use the operator $\Pi_{\FESpaceHypComp}^{\mathrm{L}}$ in Section
\ref{Sec:art_resist} to define an artificial resistivity.

\subsection{Source-system scheme}\label{Sec:NumericalSource}

The first step is elucidating a proper variational formulation for Operator \#2.
We start by multiplying \eqref{OperatorTwoVel} and \eqref{OperatorTwoH} by
smooth vector-valued test functions, $\veltest$ and $\Htest$ respectively,
and integrate by parts to obtain:
\begin{align}
\label{preWeakForm}
\begin{aligned}
 &\int_{\domain} \rho \partial_t \vel \cdot \veltest  \dx
- \mu \int_{\domain}  (\curl{}\Hfield \times \Hfield)\cdot \veltest \dx
= 0 \\
&\int_{\domain} \mu \partial_t \Hfield \cdot \Htest
+ \mu (\curl{}\Htest \times \Hfield) \cdot \vel
+ \int_{\domain} \resist \curl{}\Hfield \cdot \curl{}\Htest
+ \tfrac{\mu d_i}{\rho} (\curl{}\Hfield \times \Hfield) \cdot \curl{}\Htest
+ \dx \\
&\ \ \ + \int_{\partial\domain}
\big[- \mu (\vel \times \Hfield) + \resist\curl{}\Hfield + \tfrac{\mu
d_i}{\rho} (\curl{}\Hfield \times
\Hfield) \big] \cdot (\Htest \times \normal) \ds
= 0 \, .
\end{aligned}
\end{align}
Therefore, for the specific case when we use boundary conditions $\Hfield \times
\normal \equiv 0$ on $\partial\domain$ we propose the fully discrete scheme:
find $\{\vel_h^{n+1}, \Hfield_h^{n+1}\} \in \mathbb{V}_h^3 \times
\FESpaceHtangent$ such that
\begin{align}
\label{SourceSchemeVar}
\left\{
\begin{aligned}
&\langle \rho_h^n (\vel_h^{n+1} - \vel_h^n), \veltest_h  \rangle
- \dt_n \mu ( (\curl{}\Hfield_h^{n+\frac{1}{2}} \times
\Hfield_h^{n+\frac{1}{2}})
, \veltest_h )_{\Ltwo}
= \bzero \\
&\mu (\Hfield_h^{n+1} - \Hfield_h^n , \Htest_h)_{\Ltwo}
+ \dt_n \mu ((\curl{}\Htest_h \times \Hfield_h^{n+\frac{1}{2}}) ,
\vel_h^{n+\frac{1}{2}} )_{\Ltwo} \\
& \ \ \ + \dt_n ( \resist_h^n \curl{}\Hfield_h^{n+\frac{1}{2}} ,
\curl{}\Htest_h)_{\Ltwo}
\\
&\ \ \ + \dt_n \mu d_i \big(\tfrac{1}{\rho_h^n}
(\curl{}\Hfield_h^{n+\frac{1}{2}}
\times \Hfield_h^{n+\frac{1}{2}}), \curl {}\Htest_h \big)_{\Ltwo}
= \bzero
\end{aligned}
\right.
\end{align}
%
% \todo[inline]{do you have to put $L^2$ in the inner product?}
% Not really, because indeed the L^2 inner product might even be defined.
% But at this point in time this is just cosmetic.
%
for all $\{\veltest_h, \Htest_h\} \in \mathbb{V}_h^3 \times
\FESpaceHtangent$, where $\vel_h^{n+\frac{1}{2}} := \frac{1}{2} (\vel_h^{n}
+ \vel_h^{n+1})$, $\Hfield_h^{n+\frac{1}{2}} := \frac{1}{2} (\Hfield_h^{n} +
\Hfield_h^{n+1} )$, and $\resist_h^n \geq r$ is the finite element function
associated to the resistivity. For the purposes of this section, the reader
only needs to know $\resist_h^n(\xcoord)$ is purely explicit, meaning:
$\resist_h^n(\xcoord)$ will be a function of solution fields from the previous
time steps. More precisely, it will depend on $\{\rho_h^{n}, \vel_h^{n},
\Hfield^{n}\}$, $\{\rho_h^{n-1}, \vel_h^{n-1}, \Hfield^{n-1}\}$, and the
time-step size $\dt_{n-1}$, see Sections \ref{Sec:ArtResMotiv} and
\ref{Sec:art_resist} for more details. Note that unknown field in
\eqref{SourceSchemeVar} is $\vel_h^{n+1}(\xcoord) = \sum_{i \in
\HypVertices} \HypBasisComp_i(\xcoord) \vel_i^{n+1}$. Once we solve for
$\vel_h^{n+1}(\xcoord)$ we have to define the function:
\begin{align}
\label{SourceSchemeChangeVar}
\mom_h^{n+1}(\xcoord)
:=
\sum_{i \in \HypVertices} \HypBasisComp_i(\xcoord) \mom_i^{n+1}
\ \ \text{with} \ \
\mom_i^{n+1} := \rho_i^n \vel_i^{n+1}
\end{align}
Finally, we compute the total mechanical energy at each node as:
\begin{align}
\label{SourceSchemeTotme}
\totme_i^{n+1}  =
\totme_i^{n}
+ \frac{1}{2} \frac{|\mom_i^{n+1}|^2}{\rho_i^n}
- \frac{1}{2} \frac{|\mom_i^{n}|^2}{\rho_i^n}
+ \dt_n J_i^{n+\frac{1}{2}}
\ \ \text{for all }i \in \HypVertices ,
\end{align}
where
\begin{align}
\label{AvgJouleHeat}
J_i^{n+\frac{1}{2}}
:=
\frac{1}{m_i}
\int_{\domain} \resist_h^n |\curl \Hfield_h^{n+\frac{1}{2}}|^2
\phi_i \dx .
\end{align}
Note that, in essence, \eqref{SourceSchemeTotme} is a discrete interpretation
of \eqref{OpTwoInte}, while $J_i^{n+\frac{1}{2}}$ is just a
weighted-average or lumped-projection of the Joule heat-power. The coupled
nonlinear system \eqref{SourceSchemeVar} is meant to be solved first, once
$\{\mom_h^{n+1},
\Hfield_h^{n+1} \}$ are computed they can be plugged into
\eqref{SourceSchemeTotme} and \eqref{AvgJouleHeat} to compute
$\totme_i^{n+1}$ at each node $i \in \HypVertices$. See Section
\ref{sec:all_the_scheme} for the actual algorithmic summary. We now state the
properties satisfied by the scheme described by \eqref{SourceSchemeVar},
\eqref{SourceSchemeChangeVar}, \eqref{SourceSchemeTotme}, and
\eqref{AvgJouleHeat}.

%%%%%%%%%%%%%%%%%%%%%%%%%%%%%%%%%%%%%%%%%%%%%%%%%%%%%%%%%%%%%%%%%%%%%%%%%%%%%%%%
%%%%%%%%%%%%%%%%%%%%%%%%%%%%%%%%%%%%%%%%%%%%%%%%%%%%%%%%%%%%%%%%%%%%%%%%%%%%%%%%
%%%%%%%%%%%%%%%%%%%%%%%%%%%%%%%%%%%%%%%%%%%%%%%%%%%%%%%%%%%%%%%%%%%%%%%%%%%%%%%%

\begin{proposition}[Properties preserved by the source-update scheme]
\label{Prop:sourceupdate} Assume that the initial data is such that
\begin{align}
\label{OpTwoSourcePointAss}
\rho_i^n > 0 \  , \ \
e_i^n := \tfrac{\totme_i^n}{\rho_i^n} - \tfrac{1}{2} |\vel_i^n|^2 > 0
\ \ \text{and} \ \
\theta_i^n := [\tfrac{\partial}{\partial e} s(\tfrac{1}{\rho_i^n}, e_i)]^{-1} >
0 \ , \ \
\end{align}

The scheme \eqref{SourceSchemeVar}-\eqref{AvgJouleHeat} satisfies the energy
estimate: \\
\begin{align}
\label{EnergyKineticMagnetic}
\begin{gathered}
\sum_{i \in \HypVertices} m_i  \big(\tfrac{1}{2}\rho_i^n|\vel_i^{n+1}|^2\big)
+ \tfrac{\mu}{2} \|\Hfield_h^{n+1}\|_{\Ltwo}^2
+ \dt \|\sqrt{\resist_h^n} \curl{}\Hfield_h^{n+\frac{1}{2}} \|_{\Ltwo}^2 \\
=
\sum_{i \in \HypVertices} m_i \big(\tfrac{1}{2} \rho_i^n|\vel_i^{n}|^2\big)
+ \tfrac{\mu}{2} \|\Hfield_h^n\|_{\Ltwo}^2
\end{gathered}
\end{align}
which combined with \eqref{SourceSchemeChangeVar}, \eqref{SourceSchemeTotme},
and \eqref{AvgJouleHeat} implies that the total conservation of
energy:
\begin{align}
\label{SourceTotalEnergy}
\begin{gathered}
\sum_{i \in \HypVertices} m_i \totme_i^{n+1}
+ \tfrac{\mu}{2} \|\Hfield_h^{n+1}\|_{\Ltwo}^2
=
\sum_{i \in \HypVertices} m_i \totme_i^{n}
+ \tfrac{\mu}{2} \|\Hfield_h^n\|_{\Ltwo}^2 .
\end{gathered}
\end{align}
We have that the following nodewise/pointwise properties hold true:
\begin{align}
\label{SourcePointwiseProp}
\begin{gathered}
\rho_i^{n+1} = \rho_i^n \ , \ \
e(\state_i^{n+1}) \geq e(\state_i^{n}) \ , \ \
\theta(\state_i^{n+1}) \geq \theta(\state_i^{n}) \ , \\
s(\state_i^{n+1}) \geq s(\state_i^{n}) \ , \ \
\eta(\state_i^{n+1}) \leq \eta(\state_i^{n}).
\end{gathered}
\end{align}
Finally, we also have the preservation of the following involution property:
\begin{align}
\label{involutionProp}
(\Hfield_h^{n+1}, \nabla \EpotTest_h)_{\Ltwo} = (\Hfield_h^{n}, \nabla
\EpotTest_h)_{\Ltwo} \ \ \text{for all }\EpotTest_h \in\FESpacePotZero .
\end{align}
\end{proposition}
\begin{proof} The proof of \eqref{EnergyKineticMagnetic} follows by taking
$\veltest_h = \vel_h^{n+\frac{1}{2}}$ and $\Htest_h =
\Hfield_h^{n+\frac{1}{2}}$ in \eqref{SourceSchemeVar} and adding both
equations:
\begin{align*}
% \begin{gathered}
&\sum_{i \in \HypVertices} \tfrac{1}{2}m_i |\vel_i^{n+1}|^2
+ \tfrac{\mu}{2}\|\Hfield_h^{n+1}\|_{L^2(\domain)}^2 \\
& \ \ \ \ \ \ \ =
\sum_{i \in \HypVertices} \tfrac{1}{2}m_i |\vel_i^{n}|^2
+ \tfrac{\mu}{2} \|\Hfield_h^{n}\|_{L^2(\domain)}^2
- \dt_n \| \sqrt{\resist_h^n}
\curl{}\Hfield_h^{n+\frac{1}{2}}\|_{L^2(\domain)}^2.
% \end{gathered}
\end{align*}
Now, multiplying \eqref{SourceSchemeTotme} by $m_i$ and adding for
all $i$ in $\HypVertices$ we obtain:
\begin{align}
\label{InternalEnergyUpdate}
\sum_{i \in \HypVertices} m_i \Big( \totme_i^{n+1}
- \frac{1}{2} \frac{|\mom_i^{n+1}|^2}{\rho_i^n} \Big)
=
\sum_{i \in \HypVertices} m_i \Big( \totme_i^{n}
- \frac{1}{2} \frac{|\mom_i^{n}|^2}{\rho_i^n} \Big)
+ \dt_n m_i J_i^{n+\frac{1}{2}}.
\end{align}
Adding \eqref{EnergyKineticMagnetic} to \eqref{InternalEnergyUpdate}, then
\eqref{SourceTotalEnergy} follows as a consequence of
definition \eqref{AvgJouleHeat}. More precisely, definition
\eqref{AvgJouleHeat} implies that:
\begin{align*}
\|\sqrt{\resist_h^n} \curl{}\Hfield_h^{n+\frac{1}{2}} \|_{\Ltwo}^2
=
\sum_{i \in \HypVertices} m_i J_i^{n+\frac{1}{2}},
\end{align*}
which follows as a consequence of the partition of unity property $\sum_{i
\in \HypVertices} \phi_{i}(\xcoord) = 1$.

Now, since $\partial_t\rho$ in the context of Operator \#2, then $\rho_i^{n+1} =
\rho_i^n$ as detailed in \eqref{SourcePointwiseProp}. Now, reorganizing
\eqref{SourceSchemeTotme} we obtain:
% \label{SourceSchemeTotme}
\begin{align*}
\underbrace{\totme_i^{n+1} - \frac{1}{2} \frac{|\mom_i^{n+1}|^2}{\rho_i^n}}_{
= \, \rho_i^{n+1} \specinte_i^{n+1}} =
\underbrace{\totme_i^{n} - \frac{1}{2} \frac{|\mom_i^{n}|^2}{\rho_i^n}}_{
= \, \rho_i^{n} \specinte_i^n}
+ \dt_n J_i^{n+\frac{1}{2}}
\ \ \text{for all }i \in \HypVertices ,
\end{align*}
Therefore, we divide the previous expression by $\rho_i^{n+1} =
\rho_i^n$ which yields:
\begin{align*}
\specinte_i^{n+1} = \specinte_i^{n}
+ \dt_n \frac{J_i^{n+\frac{1}{2}}}{\rho^n} \, .
\end{align*}
Since $J_i^{n+\frac{1}{2}} \geq 0$ by construction, see expression
\eqref{AvgJouleHeat}, and $\rho^n > 0$ by assumption
\eqref{OpTwoSourcePointAss}, we conclude that $\specinte_i^{n+1} \geq
\specinte_i^{n}$. Regarding inequalities $\theta(\state_i^{n+1})
\geq \theta(\state_i^{n})$,
$ s(\state_i^{n+1}) \geq s(\state_i^{n})$ and $\eta(\state_i^{n+1}) \leq
\eta(\state_i^{n})$ in \eqref{SourcePointwiseProp}: they follow by similar
arguments to those outlined in \eqref{PositThetaProof}, \eqref{sminProof},
\eqref{etaDissPointI} and \eqref{etaDissPointII}.

Finally, involution property \eqref{involutionProp} follows by taking
$\Htest_h := \nabla \EpotTest_h$,  with $\EpotTest_h \in\FESpacePotZero$, in
the induction equation in \eqref{SourceSchemeVar} and noting that $\curl{}\nabla
\EpotTest_h \equiv 0$.
\end{proof}

\begin{remark}[Three dimensional nature of the model] We note that
Operator \#2 is intrinsically three-dimensional: the Hall term leads to
non-zero time derivative of the third component of the magnetic field
$\Hfield$, which in turn leads to a 3-dimensional Lorentz force, and ultimately
a three-dimensional momentum.

This poses the question of how to do computations without actually committing
to a fully three-dimensional code. This issue has been considered in the
existing literature in what is usually called the 2.5-space dimensions
formulation \cite{montgomery1982, laakmann2023}. The
$2.5$-dimensional implementation assumes zero
gradients in the $z$-direction and enlarges the magnetic field with a third
component that does not need to be discretized with specialized finite element
spaces. Say, for instance, the third component could be purely nodal if we
wish. The details of the $2.5$-dimensional implementation are presented in
\ref{app:two_and_half}.
% \todo[inline]{why "what is not called"? }
\end{remark}

%%%%%%%%%%%%%%%%%%%%%%%%%%%%%%%%%%%%%%%%%%%%%%%%%%%%%%%%%%%%%%%%%%%%%%%%%%%%%%%%
%%%%%%%%%%%%%%%%%%%%%%%%%%%%%%%%%%%%%%%%%%%%%%%%%%%%%%%%%%%%%%%%%%%%%%%%%%%%%%%%
%%%%%%%%%%%%%%%%%%%%%%%%%%%%%%%%%%%%%%%%%%%%%%%%%%%%%%%%%%%%%%%%%%%%%%%%%%%%%%%%

\subsection{Hyperbolic solver: assumptions}\label{sec:hypsolver_ass}

Following the same structure of our previous work \cite{Dao2024}, we place no
particular emphasis on the actual choice of hyperbolic solver used to solve
Operator \#1. The central ideas advanced in this paper are compatible with most
of the existing numerical methods used to solve Euler's equation of gas
dynamics. With that being said, this Section is almost identical to Section 3.2
of \cite{Dao2024} and it is only provided for the sake of completeness.

Given some initial data $\state_h = [\rho_h^n, \mom_h^n, \totme_h^n]^\transp$,
a numerical approximation to the solutions of $\state(\xcoord, t) =
[\rho(\xcoord, t), \mom(\xcoord, t), \totme(\xcoord, t)]^\transp$ at time $t_n$,
we have at hand a numerical procedure to compute the updated state as
\begin{align}\label{EulerAbstractScheme}
\{\rho_h^{n+1}, \mom_h^{n+1}, \totme_h^{n+1}, \dt_n\}
&:=
\texttt{euler\_system\_update}(\{\rho_h^n, \mom_h^{n}, \totme_h^n\}),
\end{align}
where $\{\rho_h^{n+1}, \mom_h^{n+1}, \totme_h^{n+1}\}$ is
the approximate solution at time $t_n + \dt_n$. Note that as described
in \eqref{EulerAbstractScheme}, $\dt_n$ is a return argument of the procedure
$\texttt{euler\_system\_update}$. In other words,
$\texttt{euler\_system\_update}$ determines the time-step size on its own. We
may at times, need to prescribe the time-step size used by
$\texttt{euler\_system\_update}$, in such case the interface of the method
might look as:
\begin{align*}
\{\rho_h^{n+1}, \mom_h^{n+1}, \totme_h^{n+1}\}
&:=
\texttt{euler\_system\_update}(\{\rho_h^n, \mom_h^{n}, \totme_h^n, \dt_n\}),
\end{align*}
where $\dt_n$ is \emph{supplied} to $\texttt{euler\_system\_update}$.

Regarding more specific properties of the
method $\texttt{euler\_system\_update}$ we may assume it is formally
second-order accurate, and most importantly, the following structural
properties hold:
\begin{itemize}
\item[\itemizebullet] \textit{Collocated/nodal discretization.} We assume that
all the components of Euler's system \eqref{OperatorOne} are discretized in a
collocated fashion meaning
  \begin{align*}
    \rho_h(\xcoord) = \sum_{i \in \HypVertices} \rho_i
    \phi_i(\xcoord) \, , \
    \mom_h(\xcoord) = \sum_{i \in \HypVertices} \mom_i
    \phi_i(\xcoord) \, , \
    \totme_h(\xcoord) = \sum_{i \in \HypVertices} \totme_i
    \phi_i(\xcoord) \, , \
  \end{align*}
  where $\rho_i \in \mathbb{R}$, $\mom_i \in \mathbb{R}^3$, $\totme_i \in
  \mathbb{R}$, and $\{\phi_i(\xcoord)\}_{i \in \HypVertices}$ is the basis
  of the scalar-valued finite element space $\FESpaceHypComp$ defined in
  \eqref{VdefSpace}.

%   This condition rules out the use of modal basis such as
%   the Bernstein or the monomial basis.

  The use of a nodal basis plays a role in the energy-stability analysis,
  see Proof \ref{Prop:sourceupdate}. More precisely, our proof recovers an
  estimate for the kinetic energy as a discrete sum on the nodes: see for
  instance update formula \eqref{SourceSchemeTotme}, energy estimate
  \eqref{EnergyKineticMagnetic}, and Corollary \ref{corollaryTotalE}. In
  essence the proof works because the degrees of freedom are associated to
  full-fledged states $\state_i = [\rho_i,  \mom_i, \totme_i]$. It is not
  clear, at this point in time, that the same proof will work in the case of,
  for instance, a monomial basis.

%   \todo[inline]{MN: Why collocated discrerization is required? In my opinion it
%   does not matter how one compute the Euler part. Please elaborate if it is
% otherwise.}
%
%   \todo[inline]{IT: It matters for the kinetic energy: you need to recover the
% kinetic energy as a discrete sum. For that you need that the degrees of freedom
% describe full states $\state_i = [\rho_i, \mom_i, \totme_i]$. If you use a
% modal basis, say for instance the monomial basis that Chi-Wang-Shu students
% love: you can't recover the kinetic energy as a discrete sum.}

\item[\itemizebullet] \textit{Conservation of linear invariants.} In the context
of periodic boundary conditions the hyperbolic solver preserves the linear
invariants:
\begin{align}\label{consMasss}
\begin{gathered}
  \sum_{i \in \HypVertices} m_i \rho_i^{n+1} = \sum_{i \in \HypVertices}
  m_i \rho_i^{n}
  \ , \ \
  \sum_{i \in \HypVertices} m_i \mom_i^{n+1} = \sum_{i \in \HypVertices}
  m_i \mom_i^{n}
  \ , \\
  \sum_{i \in \HypVertices} m_i \totme_i^{n+1} = \sum_{i \in \HypVertices}
  m_i
  \totme_i^{n},
\end{gathered}
\end{align}
where $m_i$ was defined in \eqref{schemeMatrices}.
\item[\itemizebullet] \textit{Admissibility.} We assume that if the initial
data $\state_i^n = [\rho_i^n, \mom_i^n, \totme_i^n]^\transp$ is admissible,
meaning $\state_i^n \in \mathcal{A}$ for all $i \in \HypVertices$, where the set
$\mathcal{A}$ is defined as
\begin{align}
\label{AdmissSet}
\mathcal{A} = \big\{ \ [\rho, \mom, \totme]^\transp \in \mathbb{R}^{d+2} \ | \
\rho > 0
\text{ and } \totme - \tfrac{1}{2\rho} |\mom|^2 > 0 \ \big\} \, ,
\end{align}
then the updated state $\state_i^{n+1} = [\rho_i^{n+1}, \mom_i^{n+1},
\totme_i^{n+1}]^\transp$, as defined in \eqref{EulerAbstractScheme}, is
admissible for all $i \in \HypVertices$ as well.

We highlight that, in this paper, we use \eqref{AdmissSet} as a minimal
requirement of pointwise stability for a scheme solving Euler's equation.
Beyond the simplistic case of the ideal gas law, admissibility, as described in
\eqref{AdmissSet}, is necessary but not sufficient to guarantee hyperbolicity
\cite{Meni1989, Meni2007}. For instance, the Nobel-Abel-Stiffened-Gas Equation
of State, provided as an example in \ref{app:thermo}, requires the condition
$\specinte > q + \rho^{-1} p_{\infty} (1 - \rho b) > 0$ to guarantee
hyperbolicity. A much more stringent request than \eqref{AdmissSet} would
involve preservation of hyperbolicity, positivity of temperature and minimum
principle of the specific entropy \cite{Tadmor1986, Guermond2014}. However, we
want to avoid delving into technical aspects related to equations of state
which are not central to the main ideas advanced in this paper.

\item[\itemizebullet] \textit{Entropy dissipation inequality. }We \emph{may}
  assume that the scheme preserves a global entropy inequality, meaning
  \begin{align}\label{EulerSolverDissipation}
    \sum_{i \in \HypVertices} m_i \eta(\state_i^{n+1}) \leq \sum_{i \in
    \HypVertices} m_i
    \eta(\state_i^{n}),
  \end{align}
  in the context of periodic boundary conditions.
\end{itemize}

The hyperbolic solver used for all our computations is described in Appendix B
of our previous work \cite{Dao2024}.

\subsection{Motivations for the artificial resistivity: well-posedness of the
Newton iteration.}\label{Sec:ArtResMotiv}

In this section we provide some background and heuristics that guided the
design of the artificial viscosity advanced in Section \ref{Sec:art_resist}.
We have several motivations to introduce artificial resistivity into
the scheme. But we are primarily motivated by solvability issues found in our
early computational experiments:
\begin{itemize}
\item[1. ] \textit{Nonlinear solver robustness across several types of meshes.}
We have observed (computationally) that the scheme delivers satisfactory
results, with great performance of the Newton scheme (at most 4 Newton
iterations), when using \emph{no artificial resistivity} on uniform meshes.
However, such performance did not carry over to other types of meshes.
In the context of non-uniform meshes we have that: either Newton's method
required either too many iterations, or the Jacobian turned out to be extremely
ill-conditioned, or very small time-step sizes were required when using
\emph{no} artificial resistivity.

To the best of our knowledge, all the literature on Resistive-Hall-MHD is
limited to cartesian uniform meshes. However, our goal is to raise this standard
and having a scheme that delivers comparable performance across uniform
meshes, isotropic non-uniform meshes, anisotropic meshes, and so-called
criss-cross meshes\footnote{See for instance \cite{Boffi2010} for more
information on criss-cross meshes and their pathological behaviour.}.

\item[2. ] \textit{Invertibility of the Jacobian \eqref{NewtonIteration}}.
It is possible to add enough resistivity to \emph{guarantee} the invertibility
of the Jacobian, see for instance the coercivity
estimate \eqref{CoercivityEstimateAlt}.

\item[3. ] \textit{Taming the behaviour of advective instabilities.} The
induction equation is mostly an advective-like PDE that \emph{may} benefit from
some artificial viscosity.
\end{itemize}

We expand with some technical details in relation to bullets 1, 2 and 3
listed above.
The source system \eqref{SourceSchemeVar} is solved using Newton's method.
We derived and analyzed the coercivity of the Jacobian in
\ref{app:jacobian}. In particular, the reader might want to take a look at the
variational problem that has to be solved at each Newton iteration
\eqref{NewtonIteration}, the bilinear form associated to the Jacobian
\eqref{JacobianBilinear}, Proposition \ref{PropJacobian} states
sufficient conditions to guarantee invertibility of the Jacobian, and Remark
\ref{RemarkAltEstimate}. In order to save some time to the reader, we summarize
the findings:
\begin{itemize}
\item[(\textit{i})] Estimate \eqref{CoercivityEstimateTstep} reveals that
invertibility of the Jacobian can always be recovered by using sufficiently
small time-step size $\dt$ (\emph{without} adding any artificial viscosity).
\item[(\textit{ii})] On the other hand, estimate \eqref{CoercivityEstimateAlt}
shows that a proper combination of moderately small time-step size and
artificial resistivity can yield invertibility of the Jacobian.
\end{itemize}
These two bullets hint at two strategies at our disposal to recover
well-posedness of the Newton iteration: (\textit{i}) sheer brute force (i.e.
using tiny time-step sizes), or (\textit{ii}) a combination of both moderate
time-step sizes combined with some form of artificial viscosity. We will pursue
an approach aligned with strategy (\textit{ii}). But we will not use a
viscosity strong enough to \emph{unconditionally} guarantee invertibility of the
Jacobian. Such approach would necessarily lead to an over-diffusive first-order
scheme. Our goal is more modest: we want to add a mild (but well-informed)
regularization of the residual to improve the differentiability of the
residual and invertibility of the Jacobian.

From estimates \eqref{CoercivityEstimateTstep} and
\eqref{CoercivityEstimateAlt}, in particular
estimate \eqref{CoercivityEstimateTstep}, we gather that Jacobian invertibility
might be lost if the pointwise value of one of the following:
\begin{align}
\label{InvertVelocities}
&\ \ \ \ \ \ |\curl{}\Hfield_h(\xcoord)| \,  \text{ or }
|\vel_e| \\
\label{elec_velocity}
&\text{with} \ \vel_e := \vel - \tfrac{d_i}{\rho} \curl{}\Hfield_h(\xcoord)
\end{align}
is too large. We note in passing that $\vel_e$, as defined in
\eqref{elec_velocity}, is the electron velocity, see for instance
\cite{Goed2004}. On the other hand, the induction equation \eqref{HeqnPDE} may
be rewritten/reorganized as:
\begin{align}
\nonumber
&\partial_t \Hfield
- \curl{} \big(\vel_e \times \Hfield\big)
=
- \curl{}(\tfrac{\resist}{\mu} \curl{}\Hfield)  .
\end{align}
Therefore, we may interpret $\vel_e$ as the advective velocity of the magnetic
field.

From our previous work \cite{Dao2024}, we know that
$|\curl{}\Hfield_h(\xcoord)|$ shows up in the coercivity analysis of the
Jacobian of ideal MHD. However, in practice, the magnitude
of $|\curl{}\Hfield_h(\xcoord)|$ appears to be of no concern: we have yet to
find a test case where the Jacobian of ideal MHD turns out to be singular.
Therefore, our best guess is that, in most cases of practical interest, the
electron speed $\vel_e$ is the most likely culprit for the loss of coercivity of
the Jacobian, while coercivity-loss risk associated to
$|\curl{}\Hfield_h(\xcoord)|$ is not that important. It is also important to
remember that $\vel_e$ is large near vacuum conditions. Therefore, in the next
section we devise a first-order scheme using $|\vel_e|$ as an estimate of the
maximum wavespeed of propagation.

%%%%%%%%%%%%%%%%%%%%%%%%%%%%%%%%%%%%%%%%%%%%%%%%%%%%%%%%%%%%%%%%%%%%%%%%%%%%%%%%
%%%%%%%%%%%%%%%%%%%%%%%%%%%%%%%%%%%%%%%%%%%%%%%%%%%%%%%%%%%%%%%%%%%%%%%%%%%%%%%%
%%%%%%%%%%%%%%%%%%%%%%%%%%%%%%%%%%%%%%%%%%%%%%%%%%%%%%%%%%%%%%%%%%%%%%%%%%%%%%%%

\subsection{Artificial resistivity}\label{Sec:art_resist}

In this section we define two artificial resistivities:
\begin{itemize}
\item[-] A low-order artificial resistivity that yields very robust but
sometimes over-diffused results.
\item[-] A high-order residual-based resistivity that gets triggered in
non-smooth regions of the domain, but is, otherwise very small.
\end{itemize}
In practice these two resistivities are blended in order to localize the resistive
effects only where it is necessary. More precisely, we define the resistivity
finite element function $\resist_h^n(\xcoord)$ used in scheme
\eqref{SourceSchemeVar} as:
\begin{align*}
\resist_h^n(\xcoord) = \sum_{i \in \HypVertices}  r_i \HypBasisComp_i(\xcoord)
\in
\FESpaceHypComp \ \ \text{where} \ \
r_i = \max\{r, \min\{r_i^{\text{low}}, r_i^{\text{res}}\}\} \, ,
\end{align*}
here $r > 0$ is the value of the physical resistivity, $r_i^{\text{low}}$ the
value of a first-order artificial resistivity, and $r_i^{\text{res}}$ is a
high-order residual-based resistivity. We now explain how $r_i^{\text{low}}$
and $r_i^{\text{res}}$ are computed. The low-order resistivity is defined as
\begin{align*}
r_i^{\text{low}} = c_{\text{low}} h_i \lambda_i^{\text{low}},
\end{align*}
where $c_{\text{low}}$ is an empirical non-dimensional constant of
$\mathcal{O}(1)$, $h_i$ is a measure of the grid-size in the vicinity of the
node $\xcoord_i$, and $\lambda_i^{\text{low}}$ is an approximation to
the local
maximum speed of propagation. Here $c_{\text{low}}$ and $h_i$ can be chosen
quite loosely. Our computations use the simple choices:
\begin{align}
\label{FirstOrder00}
c_{\text{low}} = 0.25 \ \ \text{and} \ \ h_i = m_i^{1/d},
\end{align}
where $m_i = \int_{\domain} \HypBasisComp_i \dx$ is the $i$-th lumped mass
entry. While the speed of propagation $\lambda_i^{\text{low}}$ is defined as
as:
\begin{align*}
& \ \ \ \ \ \ \ \ \lambda_i^{\text{low}} = |\vel_h^e(\xcoord_i)| \ \
\text{for each }i \in \HypVertices \\
&\text{where } \vel_h^e(\xcoord_i) :=  \vel_h(\xcoord_i) -
\tfrac{d_i}{\rho}
\Pi_{\FESpaceHypComp}^{\mathrm{L}}[\curl{}\Hfield_h](\xcoord_i),
\end{align*}
which is an approximation to the electron velocity as defined in formula
\eqref{elec_velocity}. Note that the lumped $L^2$-projector
$\Pi_{\FESpaceHypComp}^{\mathrm{L}}$ is necessary since $\curl{}\Hfield_h \in
H(\text{div})$ might not have well-defined pointwise values.

In order to define the residual of the induction equation we run again into the
problem that $\curl{}\Hfield_h$ is not sufficiently smooth. We resort again to
the use of the $L^2$-projector and define the residual as:
\begin{align}
\label{ResidualDef}
\begin{aligned}
& R_h^{n}(\xcoord) = \Hfield_h^n(\xcoord) - \Hfield_h^{n-1}(\xcoord)
- \dt_{n-1} \curl{}(\vel_h^{n- \frac{1}{2}} \times \Hfield^{n- \frac{1}{2}})\\
& \ \ \ + \dt_{n -1} \curl{}(\tfrac{\resist^n}{\mu}
\Pi_{\FESpaceHypComp}^{\mathrm{L}}
[\curl{}\Hfield_h^{n - \frac{1}{2}}](\xcoord)
- \tfrac{d_i}{\rho}
\Pi_{\FESpaceHypComp}^{\mathrm{L}}[\curl{}\Hfield_h^{n-\frac{1}{2}}](
\xcoord ) \times
\Hfield_h^n(\xcoord)),
\end{aligned}
\end{align}
where lumped $L^2$-projector $\Pi_{\FESpaceHypComp}^{\mathrm{L}}$ was
defined in \eqref{LumpedL2}. Note that the residual $R_h^n(\xcoord) \in
L^2(\domain)$ belongs to neither space $[\FESpaceHypComp]^d$ nor $\FESpaceH$.
Since we want to recover point values of the residual we compute its
$L^2$-projection onto the nodal vector-valued space
$[\FESpaceHypComp]^d$:
\begin{align}
\label{ResidualProj}
\mathcal{R}_h^n(\xcoord) := \Pi_{\FESpaceHypComp}^{\mathrm{L}}[R_h^n(\xcoord)]
\end{align}
and define the re-scaled residual as:
\begin{align}
\label{ResidualNormalized}
\widehat{\mathcal{R}}_h^n(\xcoord) =
\frac{\mathcal{R}_h^n(\xcoord)}{\max_{\xcoord \in
\domain}|\Hfield_h^n(\xcoord) -
\overline{\Hfield_h^n(\xcoord)}|_{\ell^2(\mathbb{R}^d)}}
\ \text{ where } \
\overline{\Hfield_h(\xcoord)} =
\frac{1}{|\domain|} \int_{\domain}\Hfield_h^n(\xcoord) \dx.
\end{align}
Finally, we define the residual-based resistivity as:
\begin{align}
\label{ArtResistResidual}
\resist_i^{\text{res}} = c_{\text{res}} h_i^2
\widehat{\mathcal{R}}_h^{n}(\xcoord_i).
\end{align}
Here again, $c_{\text{res}}$ is a constant of $\mathcal{O}(1)$. In practice we
use $c_{\text{res}} = 1.0$ for all our computations.

%%%%%%%%%%%%%%%%%%%%%%%%%%%%%%%%%%%%%%%%%%%%%%%%%%%%%%%%%%%%%%%%%%%%%%%%%%%%%%%%
%%%%%%%%%%%%%%%%%%%%%%%%%%%%%%%%%%%%%%%%%%%%%%%%%%%%%%%%%%%%%%%%%%%%%%%%%%%%%%%%
%%%%%%%%%%%%%%%%%%%%%%%%%%%%%%%%%%%%%%%%%%%%%%%%%%%%%%%%%%%%%%%%%%%%%%%%%%%%%%%%

\subsection{Algorithmic summary.}\label{sec:all_the_scheme}

\begin{algorithm}[H]
  \caption{\texttt{momentum\_and\_h\_field\_update}($
    \{\rho_h^n, \mom_h^{n}, \Hfield_h^{n}, \dt \}$)}
  \label{VelAndHcnUpdate}
  \begin{align*}
    &\texttt{Define:}\, \vel_h^{n} := \sum_{i \in \HypVertices}
      \tfrac{\mom_i^{n}}{\rho_i^n} \phi_i \\
    &\texttt{Find:} \{\vel_h^{n+1},\Hfield_h^{n+1}\} \in \mathbb{V}_h^d \times
\FESpaceHtangent \ \texttt{such that} \\
& \ \ \ \ \ \left\{
\begin{aligned}
&\langle \rho_h^n (\vel_h^{n+1} - \vel_h^n), \veltest_h  \rangle
- \dt \mu ( (\curl{}\Hfield_h^{n+\frac{1}{2}} \times
\Hfield_h^{n+\frac{1}{2}})
, \veltest_h )_{\Ltwo}
= \bzero \\
& \mu (\Hfield_h^{n+1} - \Hfield_h^n , \Htest_h)_{\Ltwo}
+ \dt \mu ((\curl{}\Htest_h \times \Hfield_h^{n+\frac{1}{2}}) ,
\vel_h^{n+\frac{1}{2}} )_{\Ltwo} \\
& \ \ \ + \dt ( \resist_h^n \curl{}\Hfield_h^{n+\frac{1}{2}}
, \curl{}\Htest_h)_{\Ltwo} \\
& \ \ \ + \dt \mu d_i \big(\tfrac{1}{\rho_h^n}
(\curl{}\Hfield_h^{n+\frac{1}{2}}
\times \Hfield_h^{n+\frac{1}{2}}), \curl {}\Htest_h \big)_{\Ltwo}
= \bzero \\
\end{aligned}
\right. \\
&\texttt{where }\vel_h^{n+\frac{1}{2}} := \tfrac{1}{2}(\vel_h^{n} +
\vel_h^{n+1}) \texttt{ and } \Hfield_h^{n+\frac{1}{2}} :=
\tfrac{1}{2}(\Hfield_h^{n} + \Hfield_h^{n+1}). \\
&\texttt{Define:}\, \mom_h^{n+1} := \sum_{i \in \HypVertices}
(\vel_i^{n+1} \rho_i^n) \phi_i \\
&\texttt{Return:} \{\mom_h^{n+1}, \Hfield_h^{n+1}\}
\end{align*}
\textit{Comments: this algorithm corresponds with the solution of the source
sytem \eqref{SourceSchemeVar}. Here it is important to note that the actual
computational implementation requires a change of variables: note the definition
of the velocity finite element function $\vel_h^{n}$ at the very beginning.
Similarly, by the end of the algorithm we also have to return to the original
dependent variable (momentum) and define the finite element function
$\mom_h^{n+1}$, which is a return argument of the method.}
\end{algorithm}

\begin{algorithm}[H]
  \caption{\texttt{source\_update}($\{\rho_h^n, \mom_h^{n},
    \totme_h^n, \Hfield_h^{n}, \dt \}$)}
  \label{SourceCNalg}
  \begin{align}
    \nonumber
    &\{\mom_h^{n+1}, \Hfield_h^{n+1}\} :=
      \texttt{momentum\_and\_h\_field\_update}(
      \{\rho_h^n, \mom_h^{n}, \Hfield_h^{n}, \dt \}) \\
    \nonumber
    &\texttt{for } i \in \HypVertices \\
    \nonumber
    % \label{UpdateDens}
    &\ \ \ \ \dens_i^{n+1} := \dens_i^{n} \\
    & \ \ \ \ J_i^{n+\frac{1}{2}} := \frac{1}{m_i}
  \int_{\domain} \resist_h^n |\curl \Hfield_h^{n+\frac{1}{2}}|^2 \phi_i \dx \\
    %
    % \label{UpdateTotMe}
    \nonumber
    &\ \ \ \ \totme_i^{n+1} = \totme_i^{n}
                        + \frac{1}{2} \frac{|\mom_i^{n+1}|^2}{\rho_i^n}
                        - \frac{1}{2} \frac{|\mom_i^{n}|^2}{\rho_i^n}
                        + \dt \, J_i^{n+\frac{1}{2}} \\
    \nonumber
    &\texttt{end for} \\
    \nonumber
    &\texttt{Return:} \{\rho_h^{n+1}, \mom_h^{n+1}, \totme_h^{n+1},
      \Hfield_h^{n+1}\}
  \end{align}
\textit{Comments: here }\texttt{source\_update} \textit{ defines the procedure
used to approximate the whole evolution described by Operator \#2, see
\eqref{OperatorTwo}. The assignments within the for-loop correspond with the
total mechanical energy update as described in
\eqref{SourceSchemeTotme}-\eqref{AvgJouleHeat}. The method }
\texttt{momentum\_and\_h\_field\_update}
\textit{is described in Algorithm \ref{VelAndHcnUpdate}.}
\end{algorithm}

\begin{algorithm}[H]
  \caption{\texttt{hall\_mhd\_update}($\{\rho_h^n, \mom_h^{n},
    \totme_h^n, \Hfield_h^{n}\}$)}
  \label{MHDupdateAlg}
  \begin{align}
    \nonumber
    \{\rho_h^1, \mom_h^{1}, \totme_h^1, \dt_n\} &:=
    \texttt{euler\_system\_update}(\{\rho_h^n, \mom_h^{n}, \totme_h^n\}) \\
    \nonumber
    \Hfield_h^{1} &:= \Hfield_h^n \\
    \nonumber
    \{\rho_h^2, \mom_h^{2}, \totme_h^2, \Hfield_h^2\} &:=
    \texttt{source\_update}
    (\{\rho_h^1, \mom_h^{1}, \totme_h^1, \Hfield_h^{1}, 2 \dt_n\}) \\
    \nonumber
    \{\rho_h^{n+1}, \mom_h^{n+1}, \totme_h^{n+1}\}
    &:=
    \texttt{euler\_system\_update}
    (\{\rho_h^2, \mom_h^{2}, \totme_h^2, \dt_n\}) \\
    \nonumber
    \Hfield_h^{n+1} &:= \Hfield_h^2 \\
    \nonumber
    \texttt{Return}&: \{\rho_h^{n+1}, \mom_h^{n+1}, \totme_h^{n+1},
                      \Hfield_h^{n+1}\}
  \end{align}
  \textit{Comments: the procedure }\texttt{euler\_system\_update} \textit{
    represents an Euler's solver satisfying the assumptions described in Section
    \ref{sec:hypsolver_ass}. On the other hand, the implementation of}
    \texttt{source\_update} \textit{is described in Algorithm \ref{SourceCNalg}.
    Note that the time-step size $\dt_n$ is determined during the
    first call to }\texttt{\texttt{euler\_system\_update}}\textit{, then the
    method} \texttt{source\_update} \textit{and the second call to}
    \texttt{\texttt{euler\_system\_update}} \textit{have to comply with such
    time-step size. Finally, note that the return argument
    $\state_h^{n+1} = [\rho_h^{n+1}, \mom_h^{n+1}, \totme_h^{n+1},
    \Hfield_h^{n+1}]^\transp$ represents the solution at time $t^{n+1} =
    t^n + 2\dt^n$.}
\end{algorithm}

For the sake of completeness we summarize the properties satisfied by the method
$\texttt{hall\_mhd\_update}$ as defined in Algorithm \ref{MHDupdateAlg} in the
following proposition.

\begin{proposition}[Properties preserved by the method
$\texttt{hall\_mhd\_update}$] For the sake of simplicity that boundary
conditions $\mom\cdot\normal = 0$ and $\Hfield \times \normal = \bzero$ are
satisfied. Then
\begin{itemize}
\item[\itemizebullet] \textit{Energy stability.} If we assume the method used
to solve Euler's equation \\
$\texttt{euler\_system\_update}$ satisfies the
conservation assumption \eqref{consMasss}, then the output produced by method
$\texttt{hall\_mhd\_update}$ satisfies the total conservation of energy property
\begin{align*}
% \label{SourceTotalEnergy}
\begin{gathered}
\sum_{i \in \HypVertices} m_i \totme_i^{n+1}
+ \tfrac{\mu}{2} \|\Hfield_h^{n+1}\|_{\Ltwo}^2
=
\sum_{i \in \HypVertices} m_i \totme_i^{n}
+ \tfrac{\mu}{2} \|\Hfield_h^n\|_{\Ltwo}^2 .
\end{gathered}
\end{align*}
\item[\itemizebullet] \textit{Admissibility.} If the method used to
solve $\texttt{euler\_system\_update}$ satisfies the admissibility requirements
outlined in the second bullet of Section \ref{sec:hypsolver_ass}, then the
resulting solution $\texttt{hall\_mhd\_update}$ is admissible as well. That is,
\begin{align*}
\state_i^{n+1} = [\rho_i^{n+1}, \mom_i^{n+1}, \totme_i^{n+1}] \in \mathcal{A}
\ \ \text{for all}  \ i \in \HypVertices
\end{align*}
where the set $\mathcal{A}$ was defined in
\eqref{AdmissSet}.
\item[\itemizebullet] \textit{Entropy-dissipation properties.} If the method
$\texttt{euler\_system\_update}$ satisfies the entropy-dissipation inequality
\eqref{EulerSolverDissipation} then the method $\texttt{hall\_mhd\_update}$
satisfies the entropy dissipation inequality as well. In particular we have that
\begin{align*}
\sum_{i \in \HypVertices} m_i \eta(\state_i^{n+1}) \leq \sum_{i \in
\HypVertices} m_i \eta(\state_i^{n})
\end{align*}
\item[\itemizebullet] \textit{Involution constraints.} The method
$\texttt{hall\_mhd\_update}$ satisfies the following involution constraint:
\begin{align}
\label{InvolutionAgain}
(\Hfield_h^{n+1}, \nabla \EpotTest_h)_{\Ltwo} = (\Hfield_h^{n}, \nabla
\EpotTest_h)_{\Ltwo} \ \ \text{for all }\EpotTest_h \in\FESpacePotZero
\end{align}
\end{itemize}
\end{proposition}

\begin{proof} The proofs mostly boil down to invoking Proposition
\ref{Prop:sourceupdate}, the assumptions on the hyperbolic solver described
in Section \ref{sec:hypsolver_ass}, and the sequential nature of operator
splitting:
\begin{itemize}
\item[\itemizebullet] \textit{Energy stability.} It follows by a sequential
argument: just use assumption \eqref{consMasss} for the hyperbolic solver and
and the discrete energy property \eqref{SourceTotalEnergy} of Operator
\#2.
\item[\itemizebullet] \textit{Admissibility.} Again it follows by the
sequential nature of operator splitting. Regarding Operator \#1 we invoke the
assumption in the third bullet of Section \ref{sec:hypsolver_ass}: the
hyperbolic solver preserves admissibility. Regarding Operator \#2 we invoke
pointwise properties \eqref{SourcePointwiseProp} which show that
specific internal energy, temperature and specific entropy only increase
during discrete evolution of Operator \#2.
\item[\itemizebullet] \textit{Entropy-dissipation.} Let's assume
that the hyperbolic solver satisfies the entropy-dissipation property
\eqref{EulerSolverDissipation}. On the other hand, the
the algorithm $\texttt{source\_update}$ satisfies the inequality
$\eta(\state_i^{2}) \leq \eta(\state_i^{1})$, see expression
\eqref{SourcePointwiseProp}. Multiplying this inequality by $m_i$ and adding
for all $i \in \HypVertices$ we obtain that the source-update scheme satisfies:
\begin{align*}
\sum_{i \in \HypVertices} m_i \eta(\state_i^{2}) \leq \sum_{i \in
\HypVertices} m_i
\eta(\state_i^{1}) .
\end{align*}
The global entropy-dissipation property follows by a sequential argument.
\item[\itemizebullet] \textit{Involution constraints.} We start by noting that
the magnetic field $\Hfield_h$ does not get modified during the evolution of
Operator \#1. On the other hand, $\texttt{source\_update}$ preserves the
involution property \eqref{InvolutionAgain} as detailed in Proposition
\ref{Prop:sourceupdate} formula \eqref{involutionProp}. Therefore, it follows
by the sequential nature of operator splitting that the method
$\texttt{hall\_mhd\_update}$ preserves the involution constraint.
\end{itemize}
\end{proof}

%%%%%%%%%%%%%%%%%%%%%%%%%%%%%%%%%%%%%%%%%%%%%%%%%%%%%%%%%%%%%%%%%%%%%%%%%%%%%%%%
%%%%%%%%%%%%%%%%%%%%%%%%%%%%%%%%%%%%%%%%%%%%%%%%%%%%%%%%%%%%%%%%%%%%%%%%%%%%%%%%
%%%%%%%%%%%%%%%%%%%%%%%%%%%%%%%%%%%%%%%%%%%%%%%%%%%%%%%%%%%%%%%%%%%%%%%%%%%%%%%%
%%%%%%%%%%%%%%%%%%%%%%%%%%%%%%%%%%%%%%%%%%%%%%%%%%%%%%%%%%%%%%%%%%%%%%%%%%%%%%%%
%%%%%%%%%%%%%%%%%%%%%%%%%%%%%%%%%%%%%%%%%%%%%%%%%%%%%%%%%%%%%%%%%%%%%%%%%%%%%%%%
%%%%%%%%%%%%%%%%%%%%%%%%%%%%%%%%%%%%%%%%%%%%%%%%%%%%%%%%%%%%%%%%%%%%%%%%%%%%%%%%

% \newpage
\section{Numerical results}
\label{sec:numeric}
% Description, domain xyz notation, normalization, parameters (cfl, RV, P1 pol)
In this section, we demonstrate the validity and robustness of the proposed
scheme. We first verify the solver by reproducing the whistler-wave dispersion
relation and obtaining high-order convergence for the linearized equations
(Section~\ref{sec:whistler}), before validating it in the nonlinear regime using
the GEM challenge problem (Section~\ref{sec:GEM}). Finally, we present novel
simulations of the Orszag--Tang vortex (Section~\ref{sec:OT}) and conclude with
a study of mesh sensitivity and the robustness of the proposed scheme
(Section~\ref{sec:mesh}).

Throughout this section, we use $P^1$ polynomials, ideal Equation of State with
adiabatic constant $\gamma=5/3$ and artificial resistivity constants
$c_{\text{low}} = 0.25$ and $c_{\text{res}} = 1$. The CFL constant is taken 0.5
for all the simulations
except for those in Section~\ref{sec:OT}. The equations are expressed in
nondimensional form, with velocities scaled by the Alfvén speed
$v_A=\HfieldComponent_0\sqrt{\mu/\rho_0}$, pressure by
$p_0=\mu\HfieldComponent_0^2$ and lengths scaled by the ion skin depth $d_i$.
Consequently, $d_i$ denotes a non-dimensional parameter in this section.

\subsection{Whistler wave}
\label{sec:whistler}

The simplest manifestation of the Hall term is the propagation of whistler
waves. The linearized Hall resistive MHD equations admit theoretical wave
solutions under uniform density and pressure, which can be used to perform an
error convergence test. Similar tests have been carried out for the case of
pure Hall MHD, see for example \cite{Chacon2025,Bard_2026}. Here, we also take
resistivity into consideration.

More precisely, we choose a background field along the 
$\unit_0$ direction and assume a small transverse wave perturbation propagating
along
$\unit_0$:
\begin{align*}
    \Hfield = \HfieldComponent_0\,\unit_0 + \delta\Hfield,\qquad
\delta\Hfield = \bigl(0,\,\delta \HfieldComponent,\,\delta
\HfieldComponent\bigr)\,e^{i(kx_0 - \omega t)}.
\end{align*}
Inserting this ansatz into the induction and momentum equations leads to the
following
dispersion relation for a right-hand polarized whistler wave
\begin{align}
    \label{DispersionRelation}
    \omega = {\frac{\omega_H+\sqrt{\omega_H^2+4\omega_A^2}}{2}}
    - i\frac{r k^2}{2}\!\left(1 +
\frac{\omega_H}{\sqrt{\omega_H^2+4\omega_A^2}}\right) + O(r^2),
\end{align}
with $\omega_H := d_ik^2\HfieldComponent_0/\rho$ and
$\omega_A^2:=k^2\HfieldComponent^2_0/\rho$. The
corresponding wave solution
is
\begin{align}\label{whislerSolution}
\begin{aligned}
    v_{1} &= -\delta v \;e^{\Im (\omega)t}\cos(kx_0 - \Re(\omega) t), &
\HfieldComponent_{1} &=
 \ \ \delta \HfieldComponent \;e^{\Im (\omega)t}\cos(kx_0 - \Re(\omega) t), \\
    v_{2} &=  \ \ \delta v \;e^{\Im (\omega)t}\sin(kx_0 - \Re(\omega) t), &
\HfieldComponent_{2} &=
-\delta \HfieldComponent \;e^{\Im (\omega)t}\sin(kx_0 - \Re(\omega) t),
\end{aligned}
\end{align}
where $\delta v$ is related to $\delta \HfieldComponent$ through the momentum
equation as
\begin{align*}
    \displaystyle\delta v = \frac{\,k \HfieldComponent_0}{\omega\rho}\,\delta
\HfieldComponent.
\end{align*}
We follow a setup analogous to that of \cite{Daldorff_2014}, with a doubly
periodic 
domain $[-L_x, L_x]\times[-L_y, L_y]$ with $L_x=80/3$ and $L_y=20$ with a
propagation direction $\unit_0$ forming
an angle $\varphi=\arctan(4/3)$
with respect to the $x$-axis. In the computational frame, the initial
perturbation phase becomes
${\Phi(x,y) = k(x\cos\varphi + y\sin\varphi).}$ This results in the following
set of initial conditions
\begin{align*}
v_{x} &= \ \ \delta v \sin\varphi\cos\Phi, & \HfieldComponent_{x} &=
\HfieldComponent_0\cos\varphi - \delta \HfieldComponent
\sin\varphi\cos\Phi, & \rho &=\rho_0,\\
v_{y} &= -\delta v \cos\varphi\cos\Phi, & \HfieldComponent_{y} &=
\HfieldComponent_0 \sin\varphi +\delta \HfieldComponent
\cos\varphi\cos\Phi, & p &=p_0,\\
v_{z} &= \ \ \delta v \sin\Phi, & \HfieldComponent_{z} &= -\delta \HfieldComponent
\sin\Phi.
\end{align*}
We set $\rho_0=1$, $\HfieldComponent_0=0.2$, $p_0=5.12\times 10^{-4}$, $d_i =
1$, $\resist = 0.001$ and $\lambda = 2\pi/k = 32$. The initial perturbation
$\delta \HfieldComponent= 0.0001$ is taken small enough such that nonlinear
effects of the PDE are negligible. The relative $L^2$ norm error is
computed between the numerical solution and the theoretical solution
\eqref{whislerSolution} after a full period $T = 2\pi/\Re(\omega)$.

The convergence results over several mesh sizes are collected in
Table~\ref{tab:convergence} and plotted in Figure~\ref{fig:convergence} along
reference rate 2 lines, corresponding to the expected convergence for $P_1$
polynomials. The rates are computed using least squares. The method achieves
near-optimal convergence rates for all the components, successfully capturing
the characteristic dispersion relation for whistler waves.

\begin{figure}[h]
    \centering
    \includegraphics[width=0.99\linewidth]{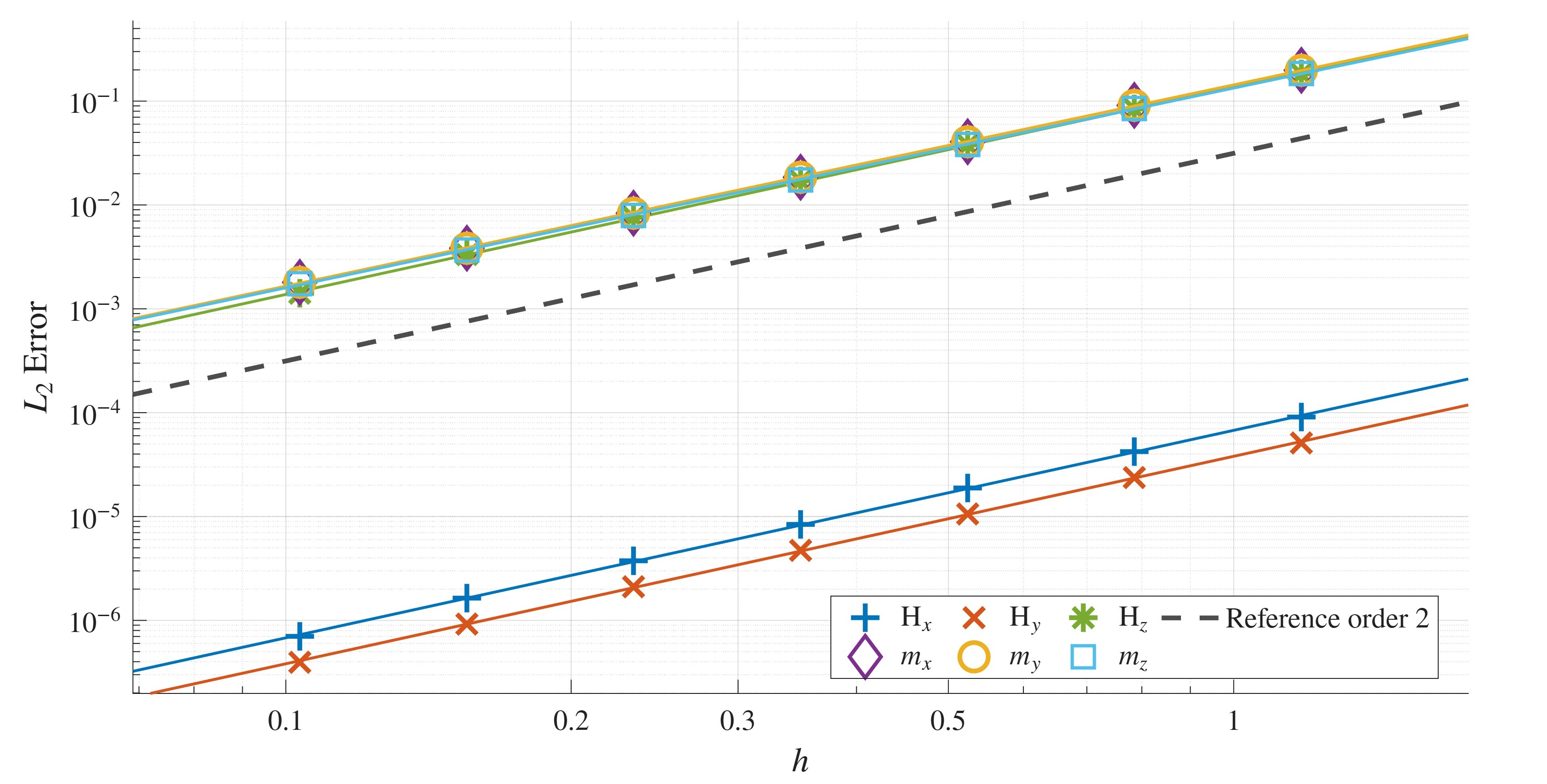}
    \caption{$L^2$ relative error for the whistler wave convergence test}
    \label{fig:convergence}
\end{figure}

\begin{table}[h]
\centering
\caption{Whistler wave $L^2$ relative errors for the components of $\Hfield$ and
$\mom$. The mesh consists of $N_x \times N_y$ rectangular cells, each split in
two elements.}
\begin{tabular}{c|ccc|ccc}
& \multicolumn{3}{c|}{$L^2$ error in $\Hfield$} 
& \multicolumn{3}{c}{$L^2$ error in $\mom$} \\
Mesh 

& $\HfieldComponent_x$ 
& $\HfieldComponent_y$ 
& $\HfieldComponent_z$ 
& $\momComponent_x$ 
& $\momComponent_y$ 
& $\momComponent_z$ \\ \hline

$64 \times 48$
& 9.10E-05
& 5.12E-05
& 1.83E-01
& 1.97E-01
& 1.97E-01
& 1.84E-01 \\

$96 \times 72$
& 4.21E-05
& 2.37E-05
& 8.44E-02
& 9.13E-02
& 9.13E-02
& 8.52E-02 \\

$144 \times 108$
& 1.87E-05
& 1.05E-05
& 3.78E-02
& 4.08E-02
& 4.08E-02
& 3.83E-02 \\

$216 \times 162$
& 8.44E-06
& 4.75E-06
& 1.71E-02
& 1.85E-02
& 1.85E-02
& 1.75E-02 \\

$324 \times 243$
& 3.74E-06
& 2.10E-06
& 7.59E-03
& 8.35E-03
& 8.35E-03
& 7.95E-03 \\

$486 \times 364$
& 1.65E-06
& 9.26E-07
& 3.33E-03
& 3.84E-03
& 3.84E-03
& 3.67E-03 \\

$729 \times 546$
& 7.04E-07
& 3.96E-07
& 1.41E-03
& 1.81E-03
& 1.81E-03
& 1.74E-03 \\ \hline

Rates
& 1.998
& 1.998
& 1.995
& 1.938
& 1.938
& 1.923

\end{tabular}
\label{tab:convergence}
\end{table}

\subsection{GEM magnetic reconnection challenge}
\label{sec:GEM}

For a nonlinear verification of the scheme, we consider the Geospace
Environmental Modeling (GEM) magnetic reconnection challenge. First introduced
by \cite{Birn2001}, it is markedly the standard and most well-studied test for
Hall MHD, as it is designed to reproduce magnetic reconnection on a perturbed
Harris sheet, where the Hall contribution is known to be critical to obtain fast
reconnection rates. Although no exact solution is known for this problem,
extensive numerical research has been carried out, and the reconnection rates
produced can be compared quantitatively.

We consider a rectangular domain $[0, L_x]\times[0, L_y]$ with $L_x=25.6$ and
$L_y=12.8$ with periodic boundary conditions for the $x$-axis. For the $y$-axis,
we use homogeneous natural boundary conditions\footnote{Sometimes referred to as
perfectly conducting boundary conditions, since it implies a zero tangential
electric field at the boundary.} by dropping the boundary terms arising from
integration by parts in the weak formulation, see expression
\eqref{preWeakForm}. The resistivity and ion-skin depth are taken to be
$r=0.005$ and $d_i=1$, respectively. We use the following set of initial
conditions: 
\begin{align*}
\rho &= \rho_0\sech^2\!\left(\frac{y-y_0}{\lambda}\right)+\rho_\infty, \quad
\quad
p = p_0 - \frac{{\HfieldComponent}_0^2}
{2}\tanh^2\!\left(\frac{y-y_0}{\lambda}\right),
\quad \quad \vel = \bzero, \\
{\HfieldComponent_x} &=
{\HfieldComponent}_0\tanh\!\left(\frac{y-y_0}{\lambda}\right)
+ \frac{\psi_0}{L_y}\pi\cos\!\left(\frac{2\pi
x}{L_x}\right)\sin\!\left(\frac{\pi(y-y_0)}{L_y}\right), \\
{\HfieldComponent_y} &=
-\frac{\psi_0}{L_x}2\pi\sin\!\left(\frac{2\pi
x}{L_x}\right)\cos\!\left(\frac{\pi(y-y_0)}{L_y}\right), \\
{\HfieldComponent_z} &= 0,
\end{align*}
where $\rho_0 = 1$, $\rho_\infty = 0.2$, ${\HfieldComponent}_0 = 1$, $p_0 =
0.6$, $\lambda = 0.5$, $y_0=6.4$ and $\psi_0 = 0.1$. Four snapshots are
illustrated in Figure~\ref{fig:GEMFine} for the out-of-plane current density
$J_z=(\curl{}\Hfield)_z$ using a mesh of 1024$\times$1024 elements. Starting
from the perturbed Harris equilibrium, the current sheet thins around the center
of the domain until it collapses into an X-line. The reconnection rate begins to
increase at $t\approx18.5$ [panel (a)], marking the onset of fast, Hall-mediated
reconnection. The precise location of the onset time is particularly
sensitive to space and time discretization. This is notable for the
cases of coarse meshes that do not fully resolve high-frequency whistler waves,
and for the case of anisotropic meshes where the onset can be delayed. After
the onset, magnetic energy is rapidly converted into kinetic and thermal energy,
driving an outflow jet away from the X-line [panel (b)]. Once the jet reaches
the periodic boundary, it thickens and interacts with itself, and the current
density at the center reverses sign [panels (c)–(d)].

The reconnected flux, $\Psi(t) = \int_{L_x/2}^{L_x}\HfieldComponent_y(x,
y=L_x/2)\;dx$ is plotted in Figure~\ref{fig:reconnexionrate} for both
Hall-resistive MHD and only resistive MHD. The rate of reconnection is
considerably larger for the Hall simulation, agreeing with the results in, for
instance, \cite{Birn2001, Strumik2017, Chacon2025}.

\begin{figure}[h]
    \centering

    % First row
    \begin{subfigure}{0.48\textwidth}
        \centering
        \includegraphics[width=\linewidth]{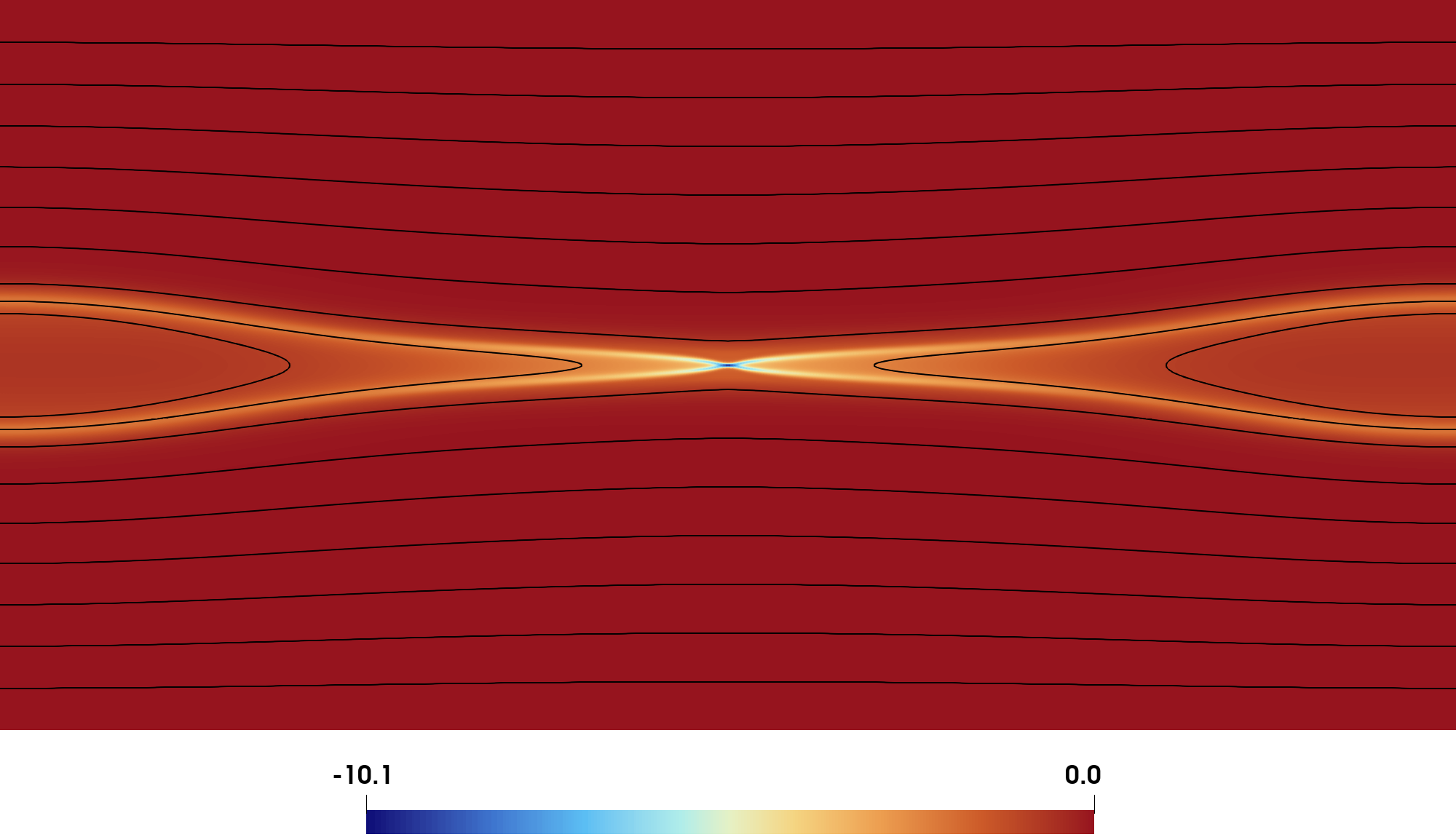}
        \caption{$t=18.5$}
        \label{fig:gem18}
    \end{subfigure}
    \hfill
    \begin{subfigure}{0.48\textwidth}
        \centering
        \includegraphics[width=\linewidth]{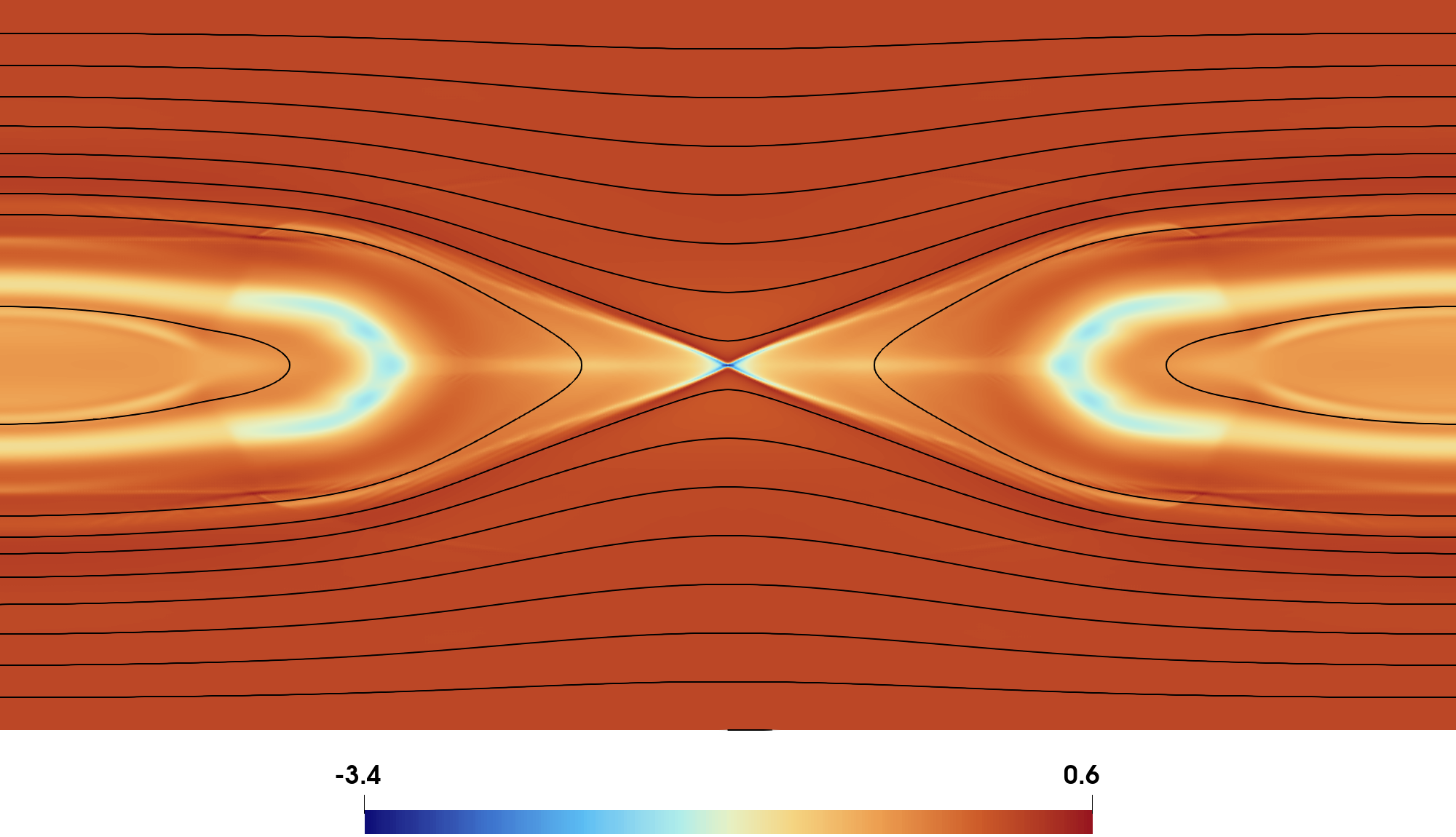}
        \caption{$t=25$}
        \label{fig:gem25}
    \end{subfigure}

    \vspace{0.5em}

    % Second row
    \begin{subfigure}{0.48\textwidth}
        \centering
        \includegraphics[width=\linewidth]{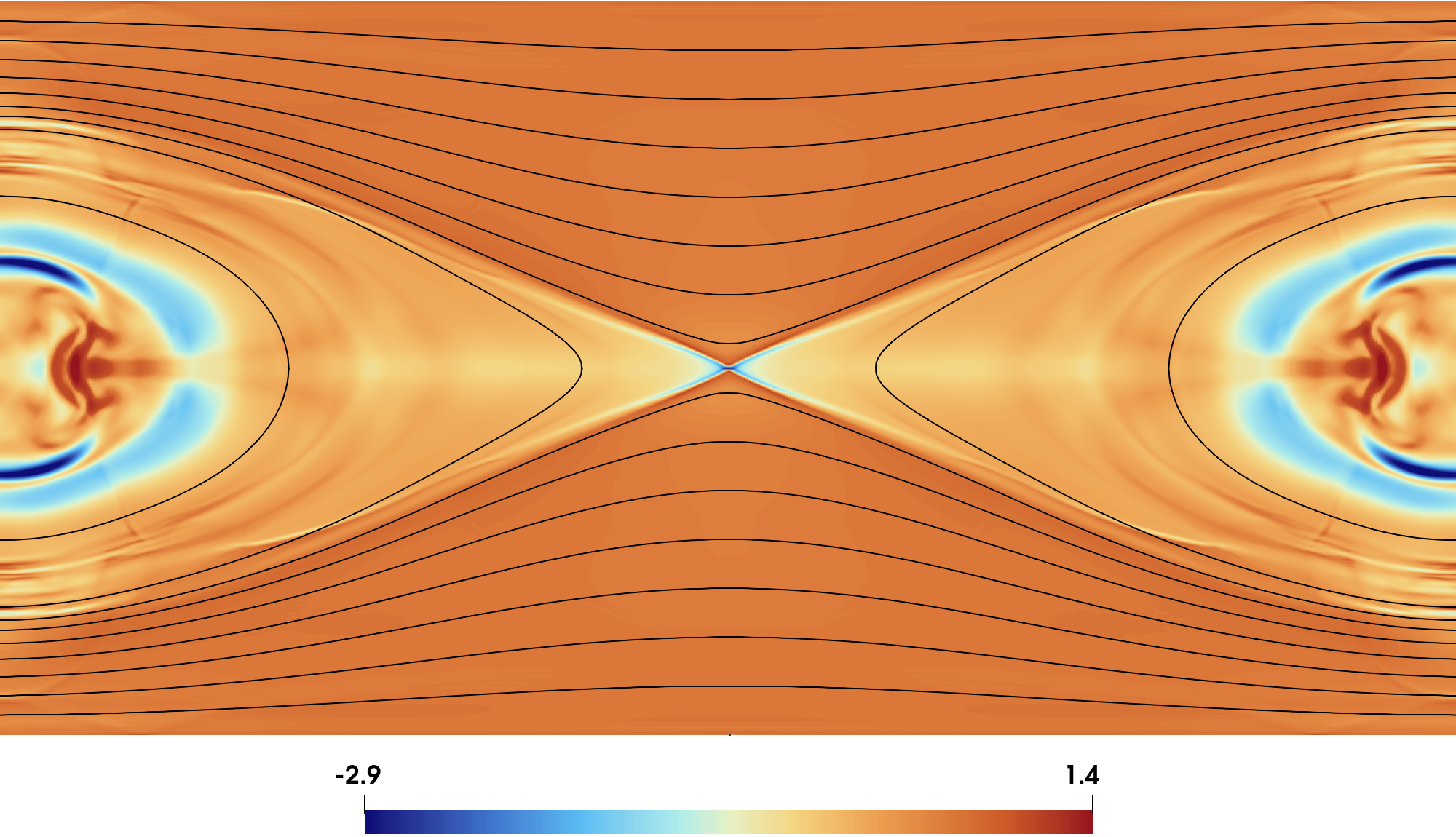}
        \caption{$t=30$}
        \label{fig:gem30}
    \end{subfigure}
    \hfill
    \begin{subfigure}{0.48\textwidth}
        \centering
        \includegraphics[width=\linewidth]{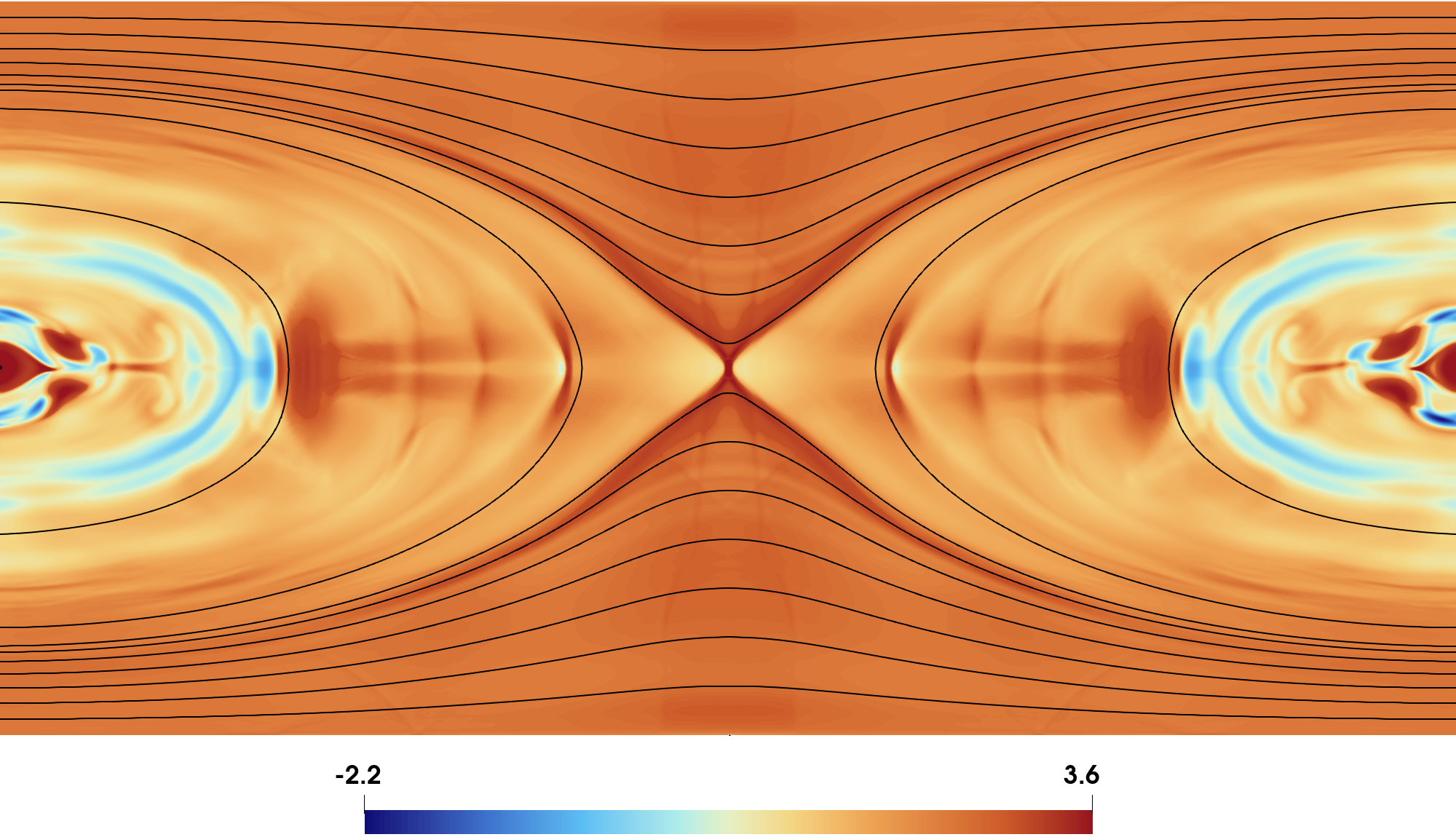}
        \caption{$t=40$}
        \label{fig:gem40}
    \end{subfigure}

    \caption{GEM magnetic reconnection challenge out-of-plane component of the
current density for $1024\times1024$ elements at times $t=18$, $25$, $30$, and
$40$. Contour plots of the magnetic field $\Hfield$ are shown in black.}
    \label{fig:GEMFine}
\end{figure}

\begin{figure}[h]
    \centering
    \includegraphics[width=0.99\linewidth]{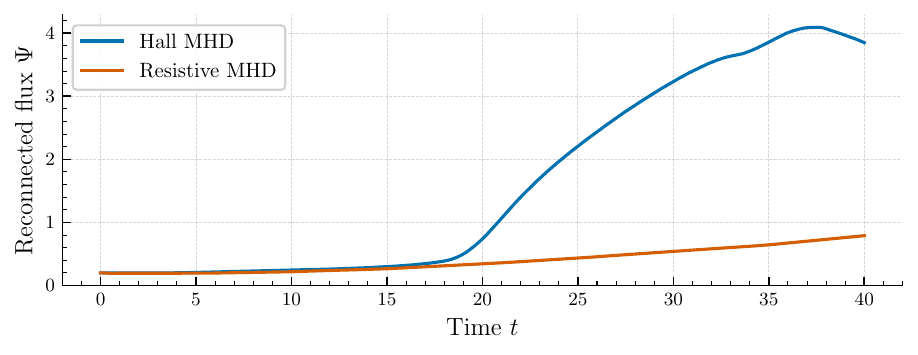}
    \caption{Reconnection rate comparison between resistive MHD with and without
the Hall term for GEM magnetic reconnection challenge.}
    \label{fig:reconnexionrate}
\end{figure}

\subsection{Orszag--Tang vortex}
\label{sec:OT}

The Orszag–Tang vortex is a standard benchmark for ideal MHD \cite{Orszag_1979},
as it evaluates the capability of a numerical scheme to resolve nonlinear shocks
and turbulence behavior. The Orszag--Tang vortex has rarely been considered in
the context of Hall MHD. While there exist studies adopting a
quasi-incompressible formulation \cite{Foldes_2023, Parashar_2009,
Stawarz_2015}, investigations in the compressible regime remain limited.
Bard et al.~\cite{Bard_2026} simulated the compressible Hall MHD Orszag--Tang
problem, but without explicit resistivity, therefore lacking a
mechanism to trigger reconnection. Multi-component kinetic frameworks
\cite{liu2017} have been shown to recover the compressible Hall-resistive
structures asymptotically in the fluid limit. However, to the best of our
knowledge, this work presents the first macroscopic simulation of the
fully compressible, resistive Hall-MHD equations for the Orszag--Tang vortex,
systematically evaluating the flow across a wide range of non-dimensional
ion skin depth values, from purely resistive MHD to strongly Hall-dominated
regimes, up to $t=1$.

The setup for the Orszag--Tang vortex test consists of a periodic
square domain $\Omega = [0,1]\times[0,1]$ with initial data
\begin{align*}
(\rho, \vel, p, \Hfield)
=
\left( \frac{25}{36\pi}, (-\sin(2\pi y),
\sin(2\pi x), 0), \frac{5}{12\pi}, \left( -\frac{\sin(2\pi y)}{\sqrt{4\pi}},
\frac{\sin(4\pi x)}{\sqrt{4\pi}},0 \right) \right),
\end{align*}
We choose a resistivity $\resist = 0.001$, sufficiently small for resistive
diffusion to remain subdominant compared to the Hall term. Previous studies have
shown that the reconnection rate is largely insensitive to the particular
mechanism responsible for breaking the frozen-in condition \cite{Birn2001}, and
the values of $d_i$ employed here are sufficient for the Hall effects to
trigger reconnection, as discussed later. Although $\resist = 0.005$ is
commonly used in the GEM setup and already produces reconnection, we have
(intentionally) adopted a smaller value in order to assess the robustness of the
solver under more demanding conditions.

\begin{figure}[h]
\centering

\includegraphics[width=130mm]{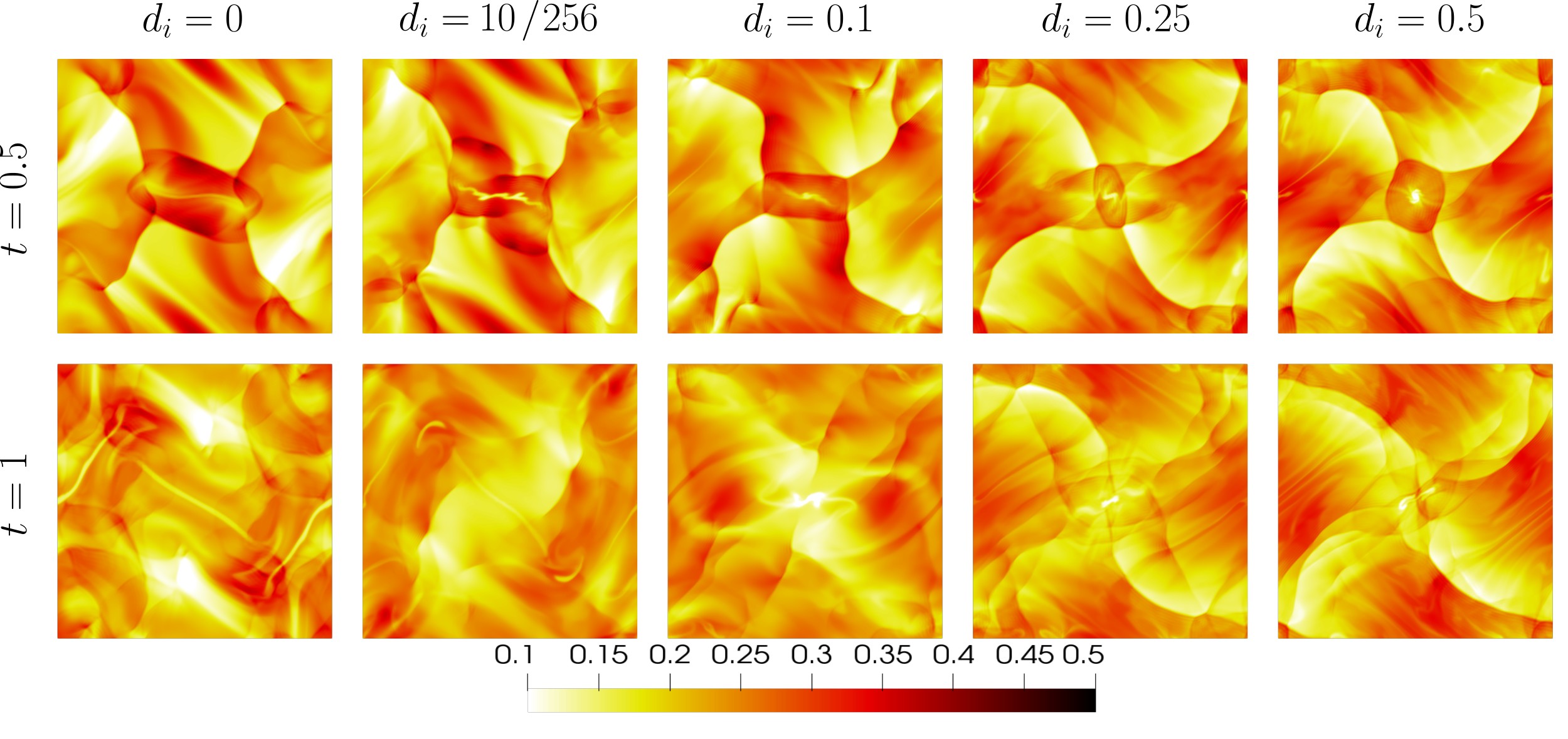}
\caption{Density distribution of the Orszag--Tang vortex on a mesh with
$256\times256$ nodal points at times $t=0.5$ and $t=1$ for ion skin
depth values $d_i=0, 10/256, 0.1, 0.25$ and $0.5$.}
\label{fig:gridDelta}
\end{figure}

The density fields at $t=0.5$ and $t=1$ are displayed in
Figure~\ref{fig:gridDelta} for a $256\times256$ nodes grid, and
several values of the ion skin depth. We include $d_i=0$ as a purely resistive
reference case, $d_i=256/10$ such that the characteristic Hall scale is
resolved by approximately 10 grid points, and $d_i=0.1,\:0.25,\:0.5$ to
illustrate a strongly Hall-dominated regime. We stress that these $d_i$ values
are relatively large and make the problem numerically demanding. For instance,
when $d_i=0.5$, the Hall term operates on a length scale comparable to half the
domain size, much larger than in typical benchmark configurations. Under these
conditions, the electron velocity is between one and two orders of magnitude
greater than the ion velocity. The Jacobian becomes less coercive, which can
impose very restrictive time-step constraints; see \ref{app:jacobian} for
details. For $d_i=0.25$ and 0.5, the CFL constant was lowered to 0.1.

The results shown in Figure~\ref{fig:gridDelta} reveal that in the purely resistive case [first column], the solution resembles a more diffusive
version of the ideal case, as magnetic reconnection occurs at a significantly
slower rate compared to Hall simulations. When
increasing the strength of $d_i$ and consequently the scale of Hall dynamics,
the solution is affected rather drastically. A low-density core forms at the
center of the domain, and a different shock distribution emerges.

To better illustrate the solution in the so-called magnetically dominated
regime, Figure~\ref{fig:OTfine} shows a fine-mesh computation with
$724\times724$ elements of the density and out-of-plane current distributions
for the highest value of the ion skin depth tested, $d_i=0.5$. A CFL of $0.05$
was required to guarantee coercivity of the Jacobian. Under these conditions,
magnetic reconnection develops almost immediately: at $t=0.02$ an X-point
geometry has already formed and the current density reaches its maximum
value of the simulation, with $|J_z|\sim 84$. As reconnection proceeds,
the magnetic energy is converted into kinetic and thermal energy, producing a
low-density core and a GEM-like outflow jet structure emanating from the
central reconnection site, as seen in $t=0.36$. Finally, by $t=1$ the flow has
transitioned into a fully turbulent state: the density field shows a rich
collection of fine-scale filaments and shear layers. The peak current density
remains comparable to that at earlier times, and its evolution stabilizes.
Overall, the evolution of the magnetic topology is qualitatively consistent with
the reconnection geometries reported by Liu and Xu~\cite{liu2017}.

\begin{figure}[h]
    \centering
    \includegraphics[width=0.99\linewidth]{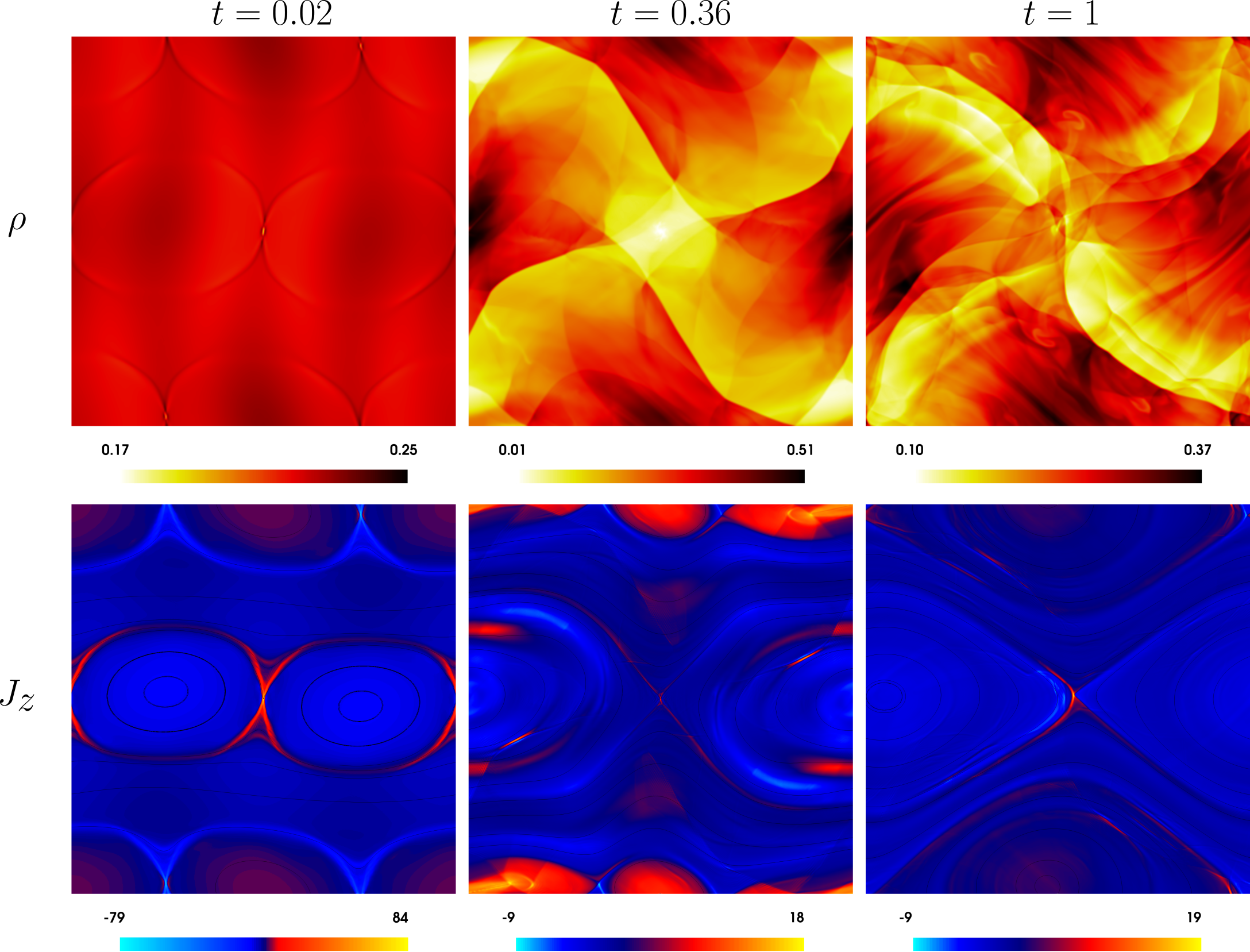}
    \caption{Snapshots of the density and out-of-plane current density for the
    Orszag--Tang simulation at $t=0.02$, $0.36$, and $1.0$ on a mesh with
    $724\times724$ elements using $d_i=0.5$. Rapid magnetic reconnection
produces a
    primary X-point and intense current sheets at early times (left). As
    reconnection proceeds, a low-density core and GEM-like outflow jets develop
    (middle). By $t=1.0$, the flow has transitioned into a fully turbulent state
    characterized by asymmetric density filaments and persistent current sheets
    (right).}
    \label{fig:OTfine}
\end{figure}

\subsection{Mesh behaviour study}
\label{sec:mesh}

Magnetic reconnection is a process that involves large, localized gradients and
steep hyperbolic fronts in reduced regions of the computational domain. In
addition, the Hall term introduces a highly asymmetric and nonlinear
contribution to the dynamics. For that reason, we assess the sensitivity of the
numerical scheme with respect to mesh-imprint artefacts by executing the GEM
magnetic reconnection challenge across four distinct mesh configurations with a
comparable number of degrees of freedom:
\begin{enumerate}
\item[\#1] \textit{Criss-cross structured mesh.} Completely symmetric. The mesh
consists of 256$\times$256 rectangular cells, each divided into two elements in
an alternating, symmetrical fashion, giving a total of 65,792 DOF for the
scalar finite element space $\FESpaceHypComp$. 
\item[\#2] \textit{Directionally-biased structured mesh.} Also built from
256$\times$256 rectangular cells, each subdivided into two triangular elements,
with all diagonals oriented to the right.
\item[\#3] \textit{Isotropic unstructured mesh.} Generated with a
Frontal-Delaunay algorithm, yielding nearly equilateral triangles throughout the
domain. The number of degrees of freedom is 65,444 DOF for the
scalar finite element space $\FESpaceHypComp$.
\item[\#4] \textit{Anisotropic unstructured mesh.} Generated with a Delaunay
algorithm under an anisotropic sizing field, so that triangles are stretched to
approximately 0.1 wide by 0.05 tall, matching the 2:1 aspect ratio used in the
structured meshes. This resulted in 72,918 DOF for the
scalar finite element space $\FESpaceHypComp$.
\end{enumerate}
The 2:1 aspect ratio is chosen deliberately: the diffusion layer that develops
along the $y$-axis requires finer resolution in that direction to capture
accurate reconnection rates, so meshes \#1, \#2, and \#4 all preserve this
anisotropy while mesh \#3, which is isotropic, does not.

Figure~\ref{fig:meshes} shows a snapshot of GEM at $t=35$ for the four described
meshes. Qualitatively, differences can be observed for each mesh. The
directionally-biased mesh \#2 breaks the symmetry of the reconnected structures,
mirroring the orientation of the mesh diagonals. The anisotropic unstructured
mesh \#4 struggles to form a single, well-defined X-line: the current sheet
appears more elongated, and two X-points form symmetrically about the domain
center before eventually merging, after which the reconnection region settles
with a slight offset from the center. We also note that the reconnected
region is (somewhat) more singular in mesh \#1 than meshes \#2, \#3 and \#4. The
singular nature of this solution caught our attention in our early
attempts at computing this solution without artificial viscosity, requiring a
significant number of Newton iterations.

The reconnected flux for each mesh is shown in
Figure~\ref{fig:meshreconnection}. Despite the differences noted above, all four
meshes eventually produce comparable reconnection rates. The main sensitivity
is found in the onset of fast reconnection. Mesh \#4 shows the most pronounced
delay, consistent with the transient double-X-point behavior: the merging of the
two X-points postpones the onset of fast reconnection relative to meshes \#1
and \#2. Mesh \#3 shows a smaller, secondary delay, which we attribute to its
isotropic triangles under-resolving the diffusion layer relative to the
anisotropic meshes, despite having a comparable overall number of degrees of
freedom.

\begin{figure}[h]
    \centering

    % First row
    \begin{subfigure}{0.47\textwidth}
        \centering
        \includegraphics[width=\linewidth]{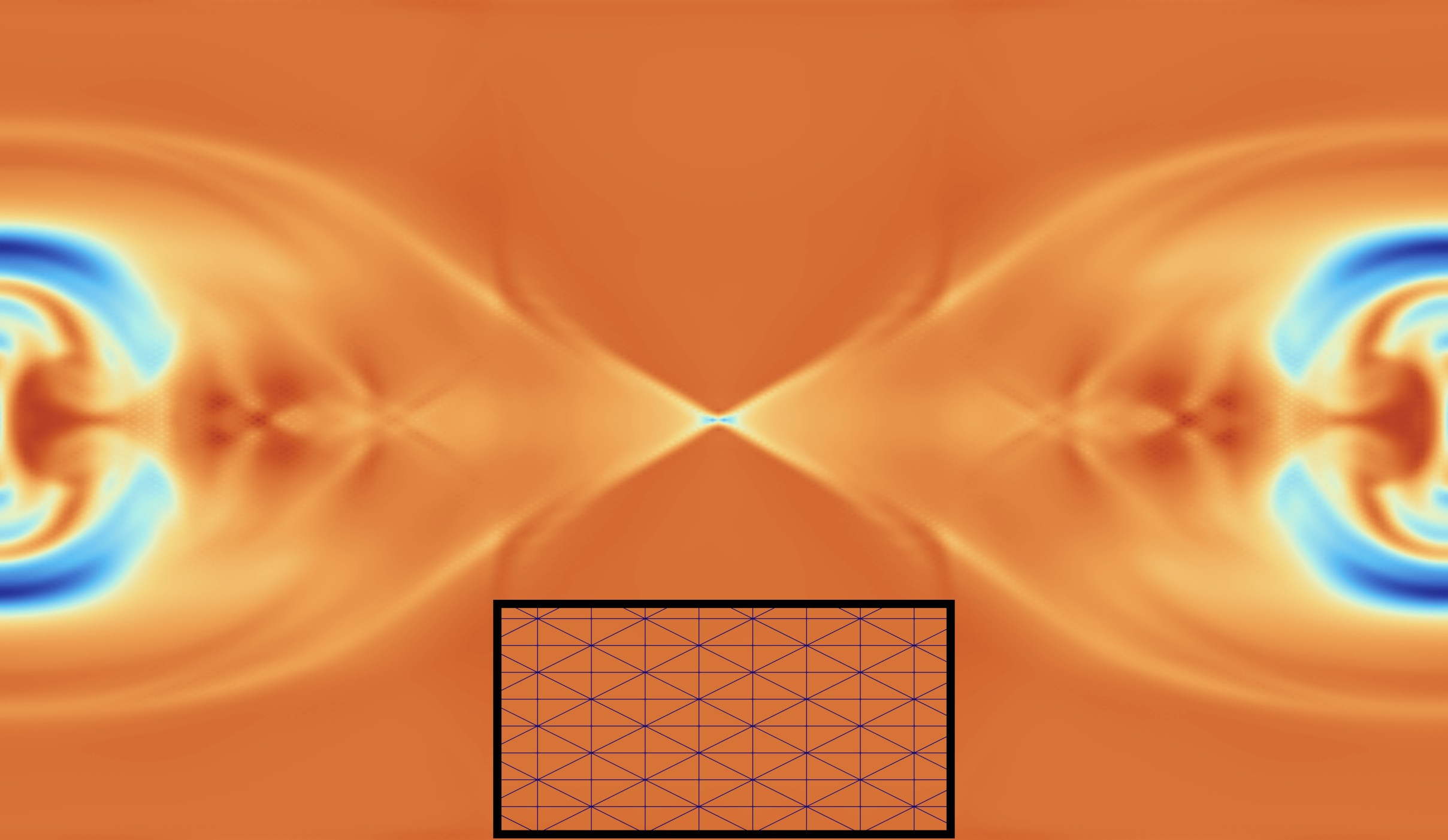}
        \caption{Criss-cross structured mesh}
        \label{fig:mesh1}
    \end{subfigure}
    \hfill
    \begin{subfigure}{0.47\textwidth}
        \centering
        \includegraphics[width=\linewidth]{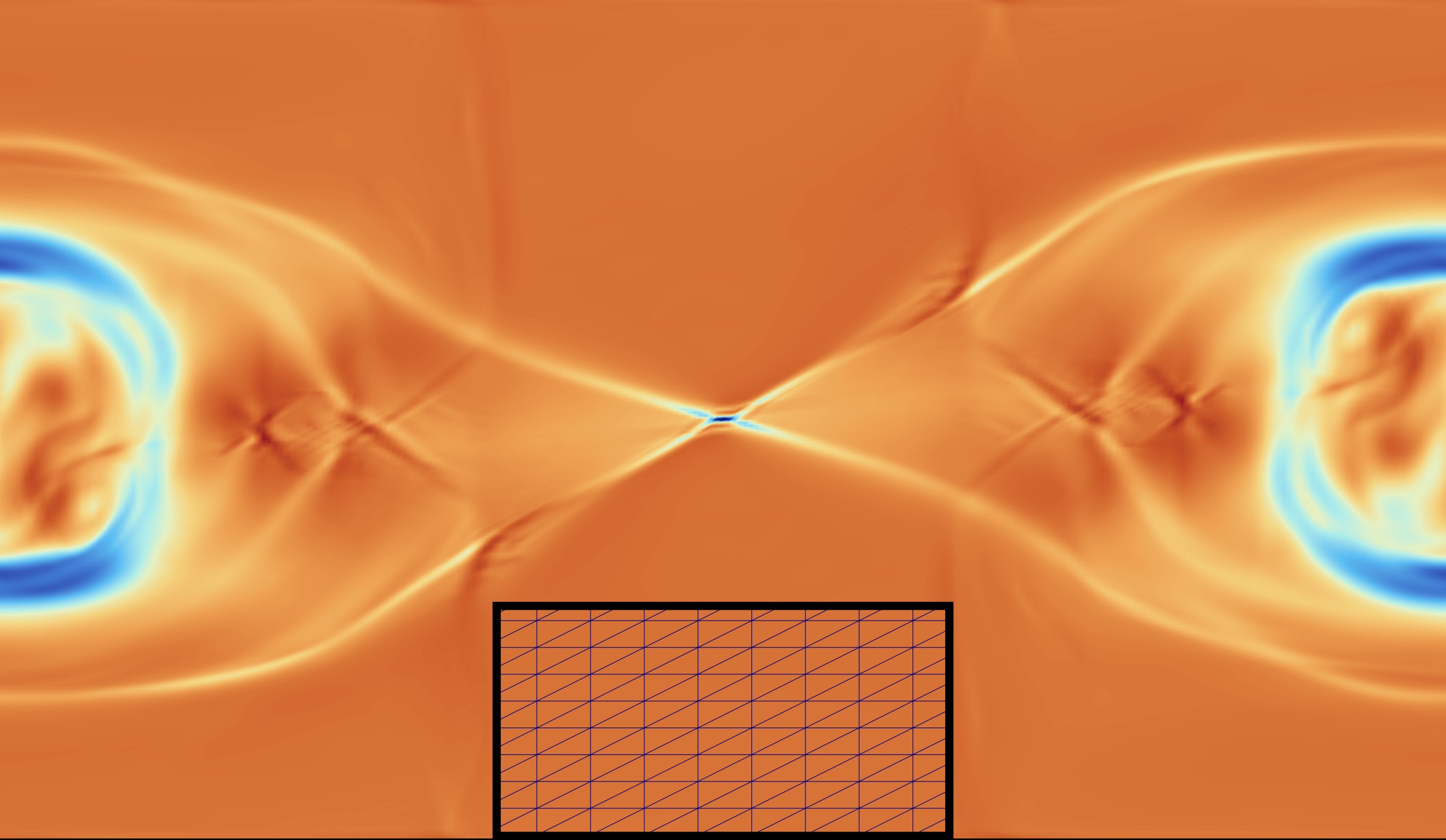}
        \caption{Directionally-biased structured mesh}
        \label{fig:mesh2}
    \end{subfigure}

    \vspace{0.5em}

    % Second row
    \begin{subfigure}{0.47\textwidth}
        \centering
        \includegraphics[width=\linewidth]{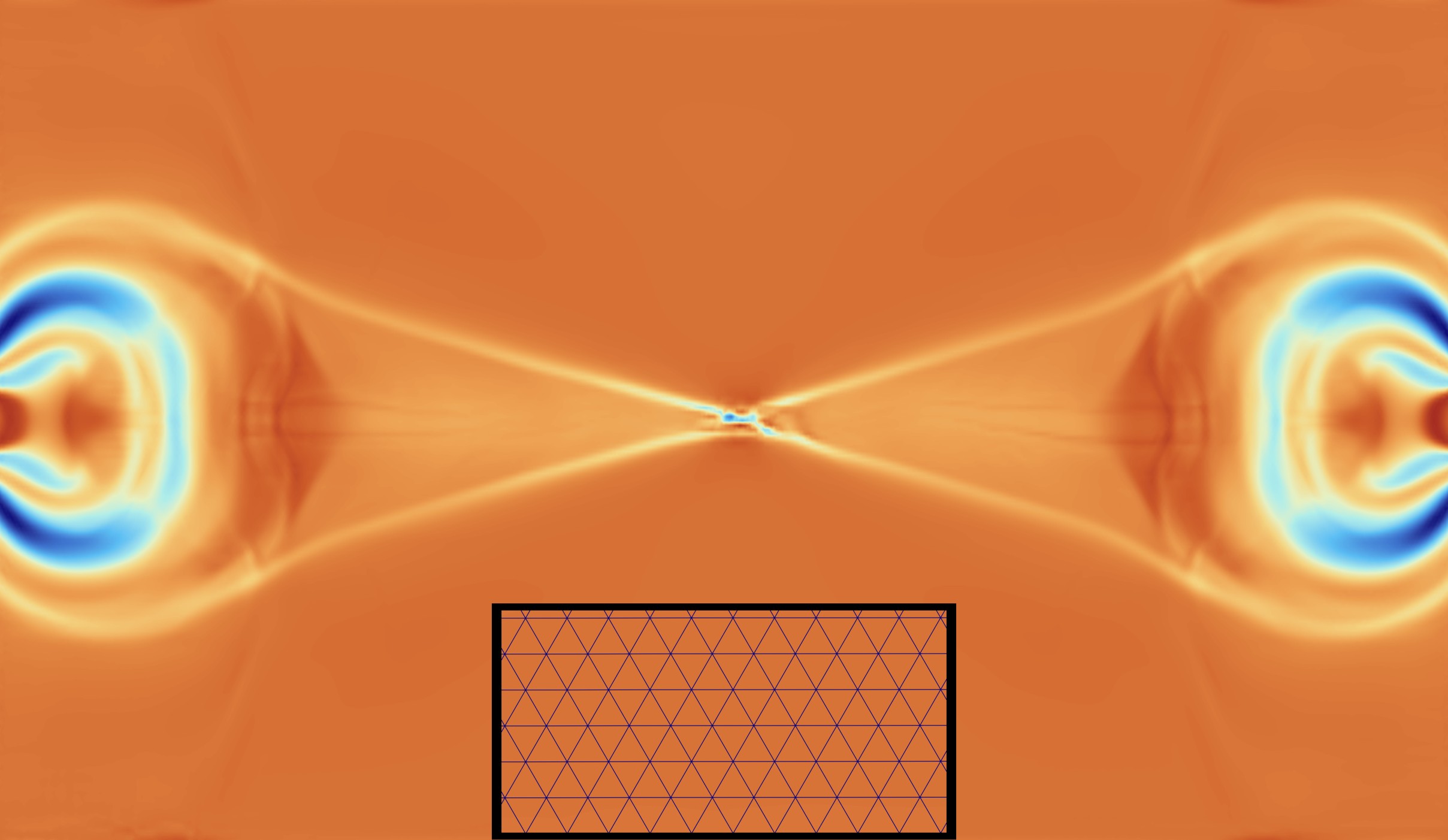}
        \caption{Isotropic unstructured mesh}
        \label{fig:mesh3}
    \end{subfigure}
    \hfill
    \begin{subfigure}{0.47\textwidth}
        \centering
        \includegraphics[width=\linewidth]{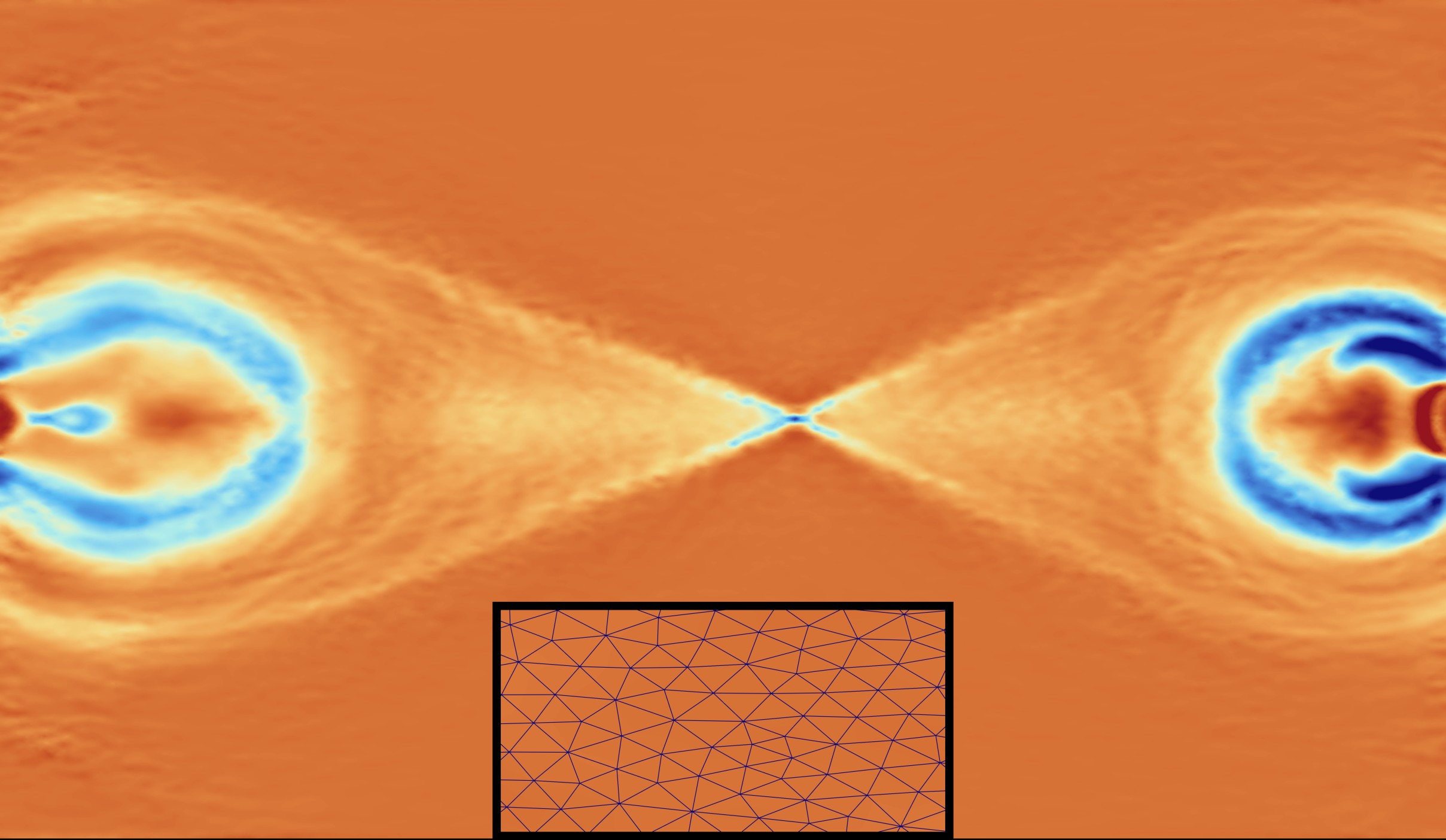}
        \caption{Anisotropic unstructured mesh}
        \label{fig:mesh4}
    \end{subfigure}

    \caption{Comparison of the GEM magnetic reconnection challenge for three
mesh topologies at $t=35$. The number of degrees of freedom used for (a) and (b) is
65,792, while 65,444 for (c) and 72,918 for (d).}
    \label{fig:meshes}
\end{figure}

\begin{figure}[h]
    \centering
    \includegraphics[width=0.99\linewidth]{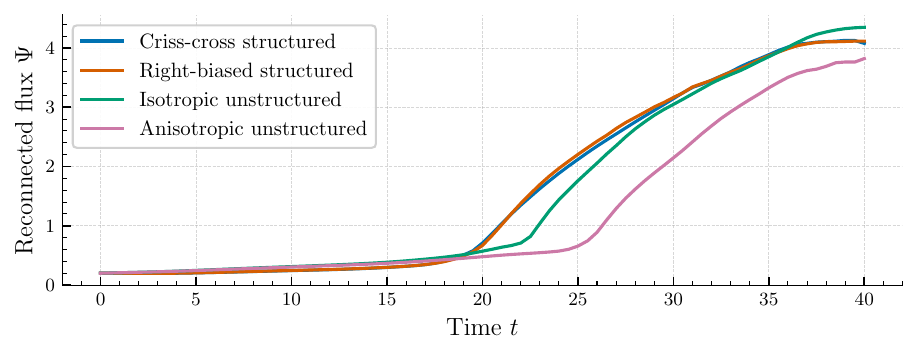}
    \caption{Reconnection rate comparison between different mesh topologies for
GEM magnetic reconnection challenge: Criss-cross structured, right-biased
structured, isotropic unstructured and anisotropic unstructured.}
    \label{fig:meshreconnection}
\end{figure}

% : isotropic structured, anisotropic structured,
% fully unstructured

%%%%%%%%%%%%%%%%%%%%%%%%%%%%%%%%%%%%%%%%%%%%%%%%%%%%%%%%%%%%%%%%%%%%%%%%%%%%%%%%
%%%%%%%%%%%%%%%%%%%%%%%%%%%%%%%%%%%%%%%%%%%%%%%%%%%%%%%%%%%%%%%%%%%%%%%%%%%%%%%%
%%%%%%%%%%%%%%%%%%%%%%%%%%%%%%%%%%%%%%%%%%%%%%%%%%%%%%%%%%%%%%%%%%%%%%%%%%%%%%%%

% \newpage
\section{Acknowledgments}

IT wants to acknowledge the continuous support of NSF grant DMS-2409841;
Sandia National Laboratories LDRD contract agreement \#1964744, award number
\#2644205; and Simons Foundation Travel Award for Mathematicians. 
MN and RV are supported by the Swedish Research Council (VR)
under grant numbers 2021-04620.

% \newpage

\appendix

\section{Thermodynamics and equations of state}\label{app:thermo}

A thermal Equation of State\footnote{The widely used acronym for the Equation
of State is EOS.} is a 2-dimensional manifold embedded in $\mathbb{R}^3$. More
precisely, such a manifold is given by the set of points
\cite{Callen1991, Lebon2008}:
\begin{align}
\label{EOSmanifold}
(\specv, e, s(\specv, e)) \subset \mathbb{R}^3
\end{align}
where $\specv = \tfrac{1}{\rho}$ is the specific volume, $e =
\tfrac{\totme}{\rho} - \frac{1}{2} |\vel|^2$ is the specific internal energy,
and $s(\specv, e):\mathbb{R}^+ \times \mathbb{R}^+ \rightarrow \mathbb{R}$ is
the specific entropy. From expression \eqref{EOSmanifold} it is tacitly
understood that $\specv$ and $e$ are the independent variables while $s$ is the
dependent variable. For any practical purpose, we may say that the specific
entropy $s(\specv, e)$ \emph{is} the EOS, since it is all you
need to describe the manifold \eqref{EOSmanifold}. For instance, for the
case of the Nobel-Abel-Stiffened-Gas, the specific entropy is given by
\cite{Saurel2016}:
\begin{align*}
s(\rho, \specinte) &= c_v \ln \Big( (\gamma - 1) \frac{\rho(e - q) +
p_{\infty} (\rho b - 1)}{1 - \rho b} \Big) - c_v \gamma \ln
\Big(\frac{(\gamma - 1) c_v \rho}{1 - \rho b} \Big) + s_0.
\end{align*}
where $c_v > 0$, $0 \leq b < +\infty$, $q > 0 $, $p_{\infty} \in
\mathbb{R}$, and $1 < \gamma \leq \tfrac{5}{3}$. For the very specific case of
$b = 0$, $q = 0$ and $p_{\infty} = 0$, the NASG EOS becomes the well-known
ideal gas specific entropy. In broad terms, an EOS describes all
thermodynamically accessible states of the substance or fluid: i.e. not every
triple of points $(\specv,e,s) \in \mathbb{R}^+ \times \mathbb{R}^+ \times
\mathbb{R}$ represents an accessible thermodynamical state.

 The pressure formula is a direct consequence of the
Gibbs identity (an exact differential, see \cite{Callen1991}). More precisely,
we have that the Gibbs identity is given by:
\begin{align*}
 \mathrm{d}s = \frac{1}{\temp}\mathrm{d}e
 + \frac{p}{\temp} \mathrm{d}\specv
 \ \ \text{where} \ \ s = s(\specv, e) \, , \
\end{align*}
which immediately implies that:
\begin{align}
\label{PressureEpist}
\frac{\partial s}{\partial e} = \frac{1}{\theta} \ \ \text{and} \ \
\frac{\partial s}{\partial v} = \frac{p}{\temp}
\ \ \text{therefore} \ \
p = p(\specv, e) = - \rho^2 \frac{\partial s}{\partial \rho}
\Big[\frac{\partial s}{\partial e}\Big]^{-1}.
\end{align}
In this paper we assume that the pressure is computed from its corresponding
specific entropy as described in \eqref{PressureEpist}. The precise formula of
the specific entropy $s(\specv, e)$ should be compatible with basic
thermodynamic constraints. We will make use of the following
standard assumptions:
\begin{itemize}
\item[\itemizebullet] \textit{Positivity of the temperature.} We
assume that:
\begin{align}
\label{posTempAssump}
\frac{\partial s}{\partial e} = \frac{1}{\theta} > 0
\ \ \text{for every} \ \
(\specv, e) \in \mathbb{R}^+ \times \mathbb{R}^+
\ \ \text{in the domain of} \ s(\specv,e)
\end{align}
Note, that in general, the domain of the specific entropy $s(\specv,
e)$ maybe a strict subset of the positive quadrant $\mathbb{R}^+
\times \mathbb{R}^+$, see \cite{Meni1989} for more details.
\item[\itemizebullet] \textit{Thermodynamic stability.} We assume that the
specific entropy is concave with respect to $\specv$ and $e$. More precisely,
we have that the following properties should hold \cite{Callen1991, Lebon2008}:
\begin{align}
\label{convexAssump}
\frac{\partial^2 s}{\partial^2 v} \leq 0 \ , \ \
\frac{\partial^2 s}{\partial^2 e} \leq 0 \ \ \text{and} \ \
\frac{\partial^2 s}{\partial^2 v}   \frac{\partial^2 s}{\partial^2 e}
- \Big(\frac{\partial^2 s}{\partial v \partial e}\Big)^2 \geq 0
\end{align}
The condition $\tfrac{\partial^2 s}{\partial^2 e} \leq 0$ is particularly
relevant in this paper, since it implies that the monotonicity condition:
\begin{align}
\label{tempMono}
 \frac{\partial }{\partial e}\theta(\specv,e) \geq 0
\end{align}
holds true. Monotonicity condition \eqref{tempMono} implies that our equation
of state is such that an increase in specific internal energy, at constant
density, can only produce a non-negative increment of temperature.
\end{itemize}

%%%%%%%%%%%%%%%%%%%%%%%%%%%%%%%%%%%%%%%%%%%%%%%%%%%%%%%%%%%%%%%%%%%%%%%%%%%%%%%%
%%%%%%%%%%%%%%%%%%%%%%%%%%%%%%%%%%%%%%%%%%%%%%%%%%%%%%%%%%%%%%%%%%%%%%%%%%%%%%%%
%%%%%%%%%%%%%%%%%%%%%%%%%%%%%%%%%%%%%%%%%%%%%%%%%%%%%%%%%%%%%%%%%%%%%%%%%%%%%%%%

\section{Jacobian of the Newton Iteration and its
properties}\label{app:jacobian}

We may rewrite the scheme \eqref{SourceSchemeVar} as: find $\{\vel_h^{n+1},
\Hfield_h^{n+1}\} \in \FESpaceHypComp^d \times \FESpaceHtangent$ satisfying
\begin{align}
\label{SolutionNewton00}
a([\vel_h^{n+1}, \Hfield_h^{n+1}], [\veltest_h, \Htest_h]) = f([\vel_h^{n},
\Hfield_h^{n}], [\veltest_h, \Htest_h]) \ \ \text{for all} \ \
[\veltest_h, \Htest_h] \in \FESpaceHypComp^d \times \FESpaceHtangent
\end{align}
where $a([\vel_h^{n+1}, \Hfield_h^{n+1}], [\veltest_h, \Htest_h])$ is a
nonlinear map defined by:
\begin{align*}
&a([\vel_h^{n+1}, \Hfield_h^{n+1}], [\veltest_h, \Htest_h]) := \langle \rho_h^n
\vel_h^{n+1} , \veltest_h  \rangle
+ \mu (\Hfield_h^{n+1}, \Htest_h)_{\Ltwo} \\
&- \tfrac{1}{4} \dt \mu ( \curl{}\Hfield_h^{n+1} \times \Hfield_h^{n},
\veltest_h )_{\Ltwo}
 - \tfrac{1}{4} \dt \mu (\curl{}\Hfield_h^{n+1} \times \Hfield_h^{n+1},
\veltest_h )_{\Ltwo} \\
&- \tfrac{1}{4} \dt \mu (\curl{}\Hfield_h^{n} \times \Hfield_h^{n+1}, \veltest_h
)_{\Ltwo}
+ \tfrac{1}{4} \dt \mu (\curl{}\Htest_h \times \Hfield_h^{n} , \vel_h^{n+1}
)_{\Ltwo} \\
&+ \tfrac{1}{4} \dt \mu (\curl{}\Htest_h \times \Hfield_h^{n+1} , \vel_h^{n}
)_{\Ltwo}
+ \tfrac{1}{4} \dt \mu (\curl{}\Htest_h \times \Hfield_h^{n+1} , \vel_h^{n+1}
)_{\Ltwo} \\
&+ \tfrac{1}{2} \dt  (\resist_h \curl{}\Hfield_h^{n+1},\curl{}\Htest_h)_{\Ltwo}
+ \tfrac{1}{4} \mu d_i \dt  \big(\tfrac{1}{\rho_h^n} \curl{}\Hfield_h^{n}
\times \Hfield_h^{n+1}, \curl {}\Htest_h \big)_{\Ltwo} \\
&+ \tfrac{1}{4} \mu d_i \dt  \big(\tfrac{1}{\rho_h^n} \curl{}\Hfield_h^{n+1}
\times \Hfield_h^{n}, \curl {}\Htest_h \big)_{\Ltwo} \\
&+ \tfrac{1}{4} \mu d_i \dt \big(\tfrac{1}{\rho_h^n} \curl{}\Hfield_h^{n+1}
\times \Hfield_h^{n+1}, \curl {}\Htest_h \big)_{\Ltwo}
\end{align*}
while $f([\vel_h^{n},\Hfield_h^{n}], [\veltest_h, \Htest_h])$ is defined by
\begin{align*}
& f([\vel_h^{n},\Hfield_h^{n}], [\veltest_h, \Htest_h])
:=
\langle \rho_h^n
\vel_h^n, \veltest_h  \rangle + \mu (\Hfield_h^n ,
\Htest_h)_{\Ltwo} \\
&\ \ \ + \tfrac{1}{4} \dt \mu (\curl{}\Hfield_h^{n} \times
\Hfield_h^{n},\veltest_h
)_{\Ltwo}
- \tfrac{1}{4} \dt \mu (\curl{}\Htest_h \times \Hfield_h^{n},
\vel_h^{n})_{\Ltwo} \\
& \ \ \ - \tfrac{1}{2} \dt  (\resist_h \curl{}\Hfield_h^{n},
\curl{}\Htest_h)_{\Ltwo}
- \tfrac{1}{4} \mu d_i \dt \big(\tfrac{1}{\rho_h^n} \curl{}\Hfield_h^{n}
\times \Hfield_h^{n}, \curl {}\Htest_h \big)_{\Ltwo}
\end{align*}
The solution process of problem \eqref{SolutionNewton00} may be achieved using
Newton's method, which consists of computing the corrections of the $k$-th
iteration state $[\vel_h^{k}, \Hfield_h^{k}]$ as
\begin{align*}
[\vel_h^{k+1}, \Hfield_h^{k+1}] := [\vel_h^{k} + \NewtonInc \vel_h^{k},
\Hfield_h^{k+1}+ \NewtonInc \Hfield_h^{k}]
\end{align*}
where $[\NewtonInc\vel_h^{k}, \NewtonInc\Hfield_h^{k}]$ is the solution of the
following linear variational problem:
\begin{align*}
a([\vel_h^{k}, \Hfield_h^{k}], [\veltest_h, \Htest_h])
+ j([\NewtonInc\vel_h^{k}, \NewtonInc\Hfield_h^{k}], [\veltest_h, \Htest_h]) =
f([\vel_h^{n}, \Hfield_h^{n}], [\veltest_h, \Htest_h]) \, ,
\end{align*}
or equivalently reorganized as:
\begin{align}
\label{NewtonIteration}
\left\{
\begin{aligned}
&\text{find } \{\NewtonInc\vel_h^{k},
\NewtonInc\Hfield_h^{k}\} \in \FESpaceHypComp^d \times \FESpaceHtangent
\text{ satisfying} \\
& j([\NewtonInc\vel_h^{k}, \NewtonInc\Hfield_h^{k}], [\veltest_h, \Htest_h]) =
f([\vel_h^{n}, \Hfield_h^{n}], [\veltest_h, \Htest_h])
- a([\vel_h^{k}, \Hfield_h^{k}], [\veltest_h, \Htest_h]) \, .
\end{aligned}
\right .
\end{align}
Here $j([\NewtonInc\vel_h^{k}, \NewtonInc\Hfield_h^{k}], [\veltest_h,
\Htest_h])$ is the Jacobian, a bilinear form defined as:
\begin{align*}
\begin{gathered}
j([\NewtonInc\vel_h^{k}, \NewtonInc\Hfield_h^{k}], [\veltest_h, \Htest_h]) :=
g'(s)|_{s=0} \\
\text{where } g(s) =
a([\vel_h^{k} + s \NewtonInc\vel_h^{k}, \Hfield_h^{k} + s
\NewtonInc\Hfield_h^{k}], [\veltest_h, \Htest_h]).
\end{gathered}
\end{align*}
Using this definition we obtain:
\begin{align}
\label{JacobianBilinear}
\begin{aligned}
& j([\NewtonInc\vel_h^{k}, \NewtonInc\Hfield_h^{k}], [\veltest_h, \Htest_h]) :=
\langle \rho_h^n \NewtonInc \vel_h^{k} , \veltest_h \rangle
+ \mu (\NewtonInc\Hfield_h^{k}, \Htest_h )_{\Ltwo} \\
& - \tfrac{1}{4} \dt \mu (
\curl{}\NewtonInc \Hfield_h^{k} \times \Hfield_h^{n}, \veltest_h )_{\Ltwo}
- \tfrac{1}{4} \dt \mu
(\curl{}\Hfield_h^{k} \times \NewtonInc\Hfield_h^{k}, \veltest_h )_{\Ltwo} \\
& - \tfrac{1}{4} \dt \mu
(\curl{}\NewtonInc\Hfield_h^{k} \times \Hfield_h^{k}, \veltest_h )_{\Ltwo}
- \tfrac{1}{4} \dt \mu (\curl{}\Hfield_h^{n} \times \NewtonInc\Hfield_h^{k},
\veltest_h )_{\Ltwo} \\
& + \tfrac{1}{4} \dt \mu (\curl{}\Htest_h \times \Hfield_h^{n} , \NewtonInc
\vel_h^{k} )_{\Ltwo}
+ \tfrac{1}{4} \dt \mu (\curl{}\Htest_h \times \NewtonInc\Hfield_h^{k},
\vel_h^{n} )_{\Ltwo} \\
&+ \tfrac{1}{4} \dt \mu (\curl{}\Htest_h \times \Hfield_h^{k}, \NewtonInc
\vel_h^{k} )_{\Ltwo}
+ \tfrac{1}{4} \dt \mu (\curl{}\Htest_h \times \NewtonInc\Hfield_h^{k},
\vel_h^{k})_{\Ltwo} \\
& + \tfrac{1}{2} \dt
(\resist_h \curl{}\NewtonInc\Hfield_h^{k},\curl{}\Htest_h)_{\Ltwo}
+ \tfrac{1}{4} \mu d_i \dt  \big(\tfrac{1}{\rho_h^n}
\curl{}\Hfield_h^{n} \times \NewtonInc\Hfield_h^{k},
\curl {}\Htest_h \big)_{\Ltwo} \\
& + \tfrac{1}{4} \mu d_i \dt  \big(\tfrac{1}{\rho_h^n}
\curl{}\NewtonInc\Hfield_h^{k} \times \Hfield_h^{n}, \curl {}\Htest_h
\big)_{\Ltwo} \\
& + \tfrac{1}{4} \mu d_i \dt \big(\tfrac{1}{\rho_h^n} \curl{}\Hfield_h^{k}
\times \NewtonInc\Hfield_h^{k}, \curl {}\Htest_h \big)_{\Ltwo} \\
& + \tfrac{1}{4} \mu d_i \dt \big(\tfrac{1}{\rho_h^n}
\curl{} \NewtonInc\Hfield_h^{k} \times \Hfield_h^{k} ,
\curl{} \Htest_h \big)_{\Ltwo} .
\end{aligned}
\end{align}
We would like to understand the coercivity properties of the Jacobian
$j([\NewtonInc\vel_h^{k}, \NewtonInc\Hfield_h^{k}], [\veltest_h, \Htest_h])$.
Setting $[\veltest_h, \Htest_h] = [\NewtonInc\vel_h^{k},
\NewtonInc\Hfield_h^{k}]$ in the previous expression we obtain:
\begin{align*}
\begin{aligned}
& j([\NewtonInc\vel_h^{k}, \NewtonInc\Hfield_h^{k} ], [\NewtonInc\vel_h^{k},
\NewtonInc\Hfield_h^{k}]) := \\
& \langle \rho_h^n \NewtonInc \vel_h^{k} , \NewtonInc\vel_h^{k} \rangle
+ \mu (\NewtonInc\Hfield_h^{k}, \NewtonInc\Hfield_h^{k} )_{\Ltwo}
+ \tfrac{1}{2} \dt (\resist_h \curl{}\NewtonInc\Hfield_h^{k},
\curl{}\NewtonInc\Hfield_h^{k})_{\Ltwo} \\
& - \tfrac{1}{4} \dt \mu (\curl{}\Hfield_h^{n} \times \NewtonInc\Hfield_h^{k},
\NewtonInc\vel_h^{k})_{\Ltwo}
- \tfrac{1}{4} \dt \mu
(\curl{}\Hfield_h^{k} \times \NewtonInc\Hfield_h^{k}, \NewtonInc\vel_h^{k}
)_{\Ltwo} \\
& + \tfrac{1}{4} \dt \mu (\curl{}\NewtonInc\Hfield_h^{k} \times
\NewtonInc\Hfield_h^{k}, \vel_h^{n} )_{\Ltwo}
+ \tfrac{1}{4} \dt \mu (\curl{}\NewtonInc\Hfield_h^{k} \times
\NewtonInc\Hfield_h^{k}, \vel_h^{k})_{\Ltwo} \\
& + \tfrac{1}{4} \mu d_i \dt \big(\tfrac{1}{\rho_h^n}
\curl{}\Hfield_h^{n} \times \NewtonInc\Hfield_h^{k},
\curl{}\NewtonInc\Hfield_h^{k} \big)_{\Ltwo} \\
&+ \tfrac{1}{4} \mu d_i \dt \big(\tfrac{1}{\rho_h^n} \curl{}\Hfield_h^{k}
\times \NewtonInc\Hfield_h^{k}, \curl{}\NewtonInc\Hfield_h^{k} \big)_{\Ltwo}.
\end{aligned}
\end{align*}
With the aid of properties of the triple product this can be further rewritten
as:
\begin{align}
\label{CoercivityAttempt}
\begin{aligned}
& j([\NewtonInc\vel_h^{k}, \NewtonInc\Hfield_h^{k} ], [\NewtonInc\vel_h^{k},
\NewtonInc\Hfield_h^{k}]) := \\
& \langle \rho_h^n \NewtonInc \vel_h^{k} , \NewtonInc\vel_h^{k} \rangle
+ \mu (\NewtonInc\Hfield_h^{k}, \NewtonInc\Hfield_h^{k} )_{\Ltwo}
+ \tfrac{1}{2} \dt (\resist_h \curl{}\NewtonInc\Hfield_h^{k},
\curl{}\NewtonInc\Hfield_h^{k})_{\Ltwo} \\
& - \tfrac{1}{4} \dt \mu (\curl{}\Hfield_h^{n} \times \NewtonInc\Hfield_h^{k},
\NewtonInc\vel_h^{k})_{\Ltwo} \\
& - \tfrac{1}{4} \dt \mu
(\curl{}\Hfield_h^{k} \times \NewtonInc\Hfield_h^{k}, \NewtonInc\vel_h^{k}
)_{\Ltwo} \\
& + \tfrac{1}{4} \mu \dt \big((\tfrac{d_i}{\rho_h^n}
\curl{}\Hfield_h^{n} - \vel_h^{n} )\times \NewtonInc\Hfield_h^{k},
\curl{}\NewtonInc\Hfield_h^{k} \big)_{\Ltwo} \\
& + \tfrac{1}{4} \mu \dt \big((\tfrac{d_i}{\rho_h^n} \curl{}\Hfield_h^{k} -
\vel_h^{k})
\times \NewtonInc\Hfield_h^{k}, \curl{}\NewtonInc\Hfield_h^{k} \big)_{\Ltwo}
\end{aligned}
\end{align}
Clearly, the first three terms in the right-hand side of
\eqref{CoercivityAttempt} are positive. However, the last four trilinear forms
are unsigned. Despite this, it is possible to show that there always exists
a sufficiently small time-step size such that the bilinear form
$j([\NewtonInc\vel_h^{k}, \NewtonInc\Hfield_h^{k}], [\veltest_h,
\Htest_h])$ is coercive. We will need to use the norm equivalence:
\begin{align}
\label{lumpingEstimate}
c_m \|\vel_h\|_{\Ltwo} \leq \langle \vel_h, \vel_h \rangle^{\frac{1}{2}} \leq
c_M \|\vel_h\|_{\Ltwo} \ \ \text{for all }\vel_h \in \FESpaceHypComp^d
\end{align}
in order to prove this statement. The proof of \eqref{lumpingEstimate} is
standard and can be found in numerous references such as \cite{Ciarlet1978,
Bartels2015}.

\begin{proposition}\label{PropJacobian} Assume that the following holds true:
\begin{align*}
\|\curl{}\Hfield_h^{n}\|_{L^\infty(\domain)} \leq c_h \ \ &\text{and} \ \
\|\curl{}\Hfield_h^{k}\|_{L^\infty(\domain)} \leq c_h \, , \\
\|\tfrac{d_i}{\rho_h^n}
\curl{}\Hfield_h^{n} - \vel_h^{n}\|_{L^\infty(\domain)}
\leq c_{e} \ \ &\text{and} \ \
\|\tfrac{d_i}{\rho_h^n}
\curl{}\Hfield_h^{n} - \vel_h^{k}\|_{L^\infty(\domain)} \leq c_{e}
\, ,
\end{align*}
for some positive bounded constants $c_h < +\infty$ and $c_{e} < +\infty$.
Then we have that the following coercivity estimate holds:
\begin{align}
\label{CoercivityEstimateTstep}
\begin{aligned}
& j([\NewtonInc\vel_h^{k}, \NewtonInc\Hfield_h^{k} ], [\NewtonInc\vel_h^{k},
\NewtonInc\Hfield_h^{k}]) \geq
\big(c_m^2 \rho_{\text{min}}^n - \tfrac{1}{4} \dt \mu c_h \big)
\|\NewtonInc\vel_h^{k}\|_{\Ltwo}^2 \\
& \ \ \ \ \ + \mu \big(1 - \tfrac{1}{4} (c_h \dt + c_e \dt^{\frac{1}{2}})
\big) \|\NewtonInc\Hfield_h^{k}\|_{\Ltwo}^2
+ \tfrac{1}{2} \dt \big(\resist_{\text{min}} - \tfrac{1}{2} \mu c_e
\dt^\frac{1}{2} \big)
\|\curl{}\NewtonInc\Hfield_h^{k}\|_{\Ltwo}^2
\end{aligned}
\end{align}
where
\begin{align*}
\rho_{\text{min}}^n = \min_{i \in \HypVertices} \rho_i^n
\ \ \text{and} \ \
\resist_{\text{min}} = \min_{\xcoord} \resist_h(\xcoord)
\end{align*}
\end{proposition}

\begin{proof} The proof is elementary and follows using Cauchy-Schwarz and
Young's inequality estimates on the last four terms of \eqref{CoercivityAttempt}:
\begin{align*}
& \big| \tfrac{1}{4} \dt \mu (\curl{}\Hfield_h^{n} \times
\NewtonInc\Hfield_h^{k},
\NewtonInc\vel_h^{k})_{\Ltwo}\big|
\leq
\tfrac{1}{4} \dt \mu c_h \big(
\tfrac{\epsilon_1}{2} \|\NewtonInc\Hfield_h^{k}\|^2
+ \tfrac{1}{2 \epsilon_1} \|\NewtonInc\vel_h^{k}\|^2 \big) \\
&\big| \tfrac{1}{4} \dt \mu
(\curl{}\Hfield_h^{k} \times \NewtonInc\Hfield_h^{k}, \NewtonInc\vel_h^{k}
)_{\Ltwo}\big|
\leq
\tfrac{1}{4} \dt \mu c_h \big(\tfrac{\epsilon_2}{2}
\|\NewtonInc\Hfield_h^{k}\|^2
+ \tfrac{1}{2 \epsilon_2} \|\NewtonInc\vel_h^{k}\|^2 \big) \\
& |\tfrac{1}{4} \mu \dt \big((\tfrac{d_i}{\rho_h^n}
\curl{}\Hfield_h^{n} - \vel_h^{n} )\times \NewtonInc\Hfield_h^{k},
\curl{}\NewtonInc\Hfield_h^{k} \big)_{\Ltwo}|
\leq
\tfrac{1}{4} \mu \dt c_e \big(
\tfrac{\epsilon_3}{2}\|\curl{}\NewtonInc\Hfield_h^{k}\|^2
+ \tfrac{1}{2 \epsilon_3} \|\NewtonInc\Hfield_h^{k}\|^2\big)
\\
& |\tfrac{1}{4} \mu \dt \big((\tfrac{d_i}{\rho_h^n} \curl{}\Hfield_h^{k} -
\vel_h^{k})
\times \NewtonInc\Hfield_h^{k}, \curl{}\NewtonInc\Hfield_h^{k} \big)_{\Ltwo}|
\leq
\tfrac{1}{4} \mu \dt c_e \big(
\tfrac{\epsilon_4}{2}\|\curl{}\NewtonInc\Hfield_h^{k}\|^2
+ \tfrac{1}{2 \epsilon_4} \|\NewtonInc\Hfield_h^{k}\|^2\big)
\end{align*}
Then we choose:
\begin{align*}
\epsilon_1 = 1 \ , \ \
\epsilon_2 = 1 \ , \ \
\epsilon_3 = \epsilon_4 = \sqrt{\dt}
\end{align*}
Inserting these estimates, multiplied by $-1$ into the right-hand side of
\eqref{CoercivityAttempt}, using the estimate $\langle \rho_h^n \NewtonInc
\vel_h^{k} , \NewtonInc\vel_h^{k} \rangle \geq \rho_{\textit{min}}^n \langle
\NewtonInc \vel_h^{k} , \NewtonInc\vel_h^{k} \rangle$ together with lumping
estimate \eqref{lumpingEstimate}, and grouping the terms yields the result.
\end{proof}

In essence, estimate \eqref{CoercivityEstimateTstep} is telling us that there
always exists a time step size sufficiently small, such that the matrix
corresponding to the bilinear form \eqref{JacobianBilinear} is positive
definite. The last line of estimate \eqref{CoercivityEstimateTstep} also
motivates us to consider the development of artificial resistivities
proportional to $c_h$ and $c_e$. That is, an artificial resistivity
that improves local coercivity\footnote{Thereby, stability of the scheme} should
be proportional to $|\curl{}\Hfield|$ and/or $|\vel_h^{k} -
\tfrac{d_i}{\rho_h^n} \curl{}\Hfield_h^{n}|$.

% Though both wavespeeds are
% important
% Though, both bounds are important, $c_e$ is un
% upper bound on the electron velocity. We note that

Knowing that it is possible to reduce the time-step size in order to guarantee
invertibility of the Jacobian is somewhat comforting, but it's not perfectly
satisfactory. Such an approach could lead to a minuscule time-step size. It would
be interesting to know that we have other options at our disposal to enforce
invertibility of the Jacobian. In this regard, we note that estimate
\eqref{CoercivityEstimateTstep} is just a consequence of the choice of
parameters $\epsilon_1$, $\epsilon_2$, $\epsilon_3$ and $\epsilon_4$.
Ultimately, there is no optimal choice of parameters. By changing the options of
parameters; $\epsilon_1$, $\epsilon_2$, $\epsilon_3$ and $\epsilon_4$; we may
be able to arrive to a different understanding of what it takes to guarantee
invertibility of the Jacobian and stabilize the scheme. In this vein, we have
the following remark.

\begin{remark}[Alternative estimate]\label{RemarkAltEstimate} By changing the
choice of parameters; $\epsilon_1$, $\epsilon_2$, $\epsilon_3$ and $\epsilon_4$;
we can obtain an estimate that is slightly different from
\eqref{CoercivityEstimateTstep}. More precisely, if we choose $\epsilon_1 =
\epsilon_2 = \epsilon_3 = \epsilon_4 = 1$ we obtain the following (alternative)
estimate:
\begin{align}
\label{CoercivityEstimateAlt}
\begin{split}
& j([\NewtonInc\vel_h^{k}, \NewtonInc\Hfield_h^{k} ], [\NewtonInc\vel_h^{k},
\NewtonInc\Hfield_h^{k}]) \geq
\big(c_m^2 \rho_{\text{min}}^n - \tfrac{1}{4} \mu c_h \dt \big)
\|\NewtonInc\vel_h^{k}\|_{\Ltwo}^2 \\
& \ \ \ \ \ + \mu \big(1 - \tfrac{1}{4} \mu (c_h + c_e ) \dt \big)
\|\NewtonInc\Hfield_h^{k}\|_{\Ltwo}^2
+ \tfrac{1}{2} \dt \big(\resist_{\text{min}}  - \tfrac{1}{2} \mu c_e
\big) \|\curl{}\NewtonInc\Hfield_h^{k}\|_{\Ltwo}^2
\end{split}
\end{align}
Note that in this case shrinking the time step size may only have a limited
effect. More precisely, from the first and second line of
\eqref{CoercivityEstimateAlt} we gather than if we choose $\dt$ sufficiently
small we can guarantee coercivity for the terms $\|\NewtonInc\vel_h^{k}
\|_{\Ltwo}^2$ and $\|\NewtonInc\Hfield_h^{k}\|_{\Ltwo}^2$. However, from the
last line in \eqref{CoercivityEstimateAlt}, we realize that we can only obtain a
lower bound on the term $\|\curl{}\NewtonInc\Hfield_h^{k}\|_{\Ltwo}^2$ if the
resistivity is sufficiently large. More precisely the resistivity has to
satisfy the bound $\resist_{\text{min}} \geq \tfrac{1}{2} \mu c_e$. Therefore,
estimate \eqref{CoercivityEstimateAlt} hints at the idea that a combination of a
sufficiently small time-step and a sufficiently large artificial viscosity
should yield the right compromise.
\end{remark}

%%%%%%%%%%%%%%%%%%%%%%%%%%%%%%%%%%%%%%%%%%%%%%%%%%%%%%%%%%%%%%%%%%%%%%%%%%%%%%%%
%%%%%%%%%%%%%%%%%%%%%%%%%%%%%%%%%%%%%%%%%%%%%%%%%%%%%%%%%%%%%%%%%%%%%%%%%%%%%%%%
%%%%%%%%%%%%%%%%%%%%%%%%%%%%%%%%%%%%%%%%%%%%%%%%%%%%%%%%%%%%%%%%%%%%%%%%%%%%%%%%

\section{Two and a half space dimensions
implementation}\label{app:two_and_half}
The physics induced by the Hall term is intrinsically three-dimensional: even
for an initial magnetic field lying in a plane, the Hall term generates a
component normal to that plane. While the scheme described in
Section~\ref{Sec:NumericalScheme} is fully compatible with $d=3$, it is
customary in the literature to use the so-called \emph{2.5-$d$} formulation,
which we adopt for the numerical results presented in our work. It consists of a
hybrid approach between two and three dimensions that uses a planar
computational domain $\Omega\subset\mathbb{R}^2$ but keeps full three-component
vector fields $\mom$ and $\Hfield$, with vanishing out-of-plane derivative. Let
$\xcoord=(x, y)\in\Omega$ and $z$ be the out-of-plane coordinate. Then
\begin{align*}
    \Hfield(\xcoord,t) = \begin{bmatrix}
    \Hfield_{xy} \\
    \HfieldComponent_z
\end{bmatrix} = \begin{bmatrix}
    \HfieldComponent_x \\
    \HfieldComponent_y \\
    \HfieldComponent_z \end{bmatrix}, \qquad
\mom (\xcoord,t) = \begin{bmatrix}
    \mom_{xy}\\
    \momComponent_z
\end{bmatrix} = \begin{bmatrix}
    \momComponent_x \\
    \momComponent_y \\
    \momComponent_z \end{bmatrix}, \qquad \partial_z\equiv0.
\end{align*}
Here, $\Hfield_{xy}$ and $\mom_{xy}$ are 2-vectors fields representing the in-plane part, whereas $\HfieldComponent_z$ is the scalar out-of-plane part. The action of the curl and div operators becomes
\begin{align*}
    \halfCurl \Hfield :=
\begin{bmatrix}
    \partial_y\HfieldComponent_z \\
    -\partial_x\HfieldComponent_z \\
    \partial_x\HfieldComponent_y-\partial_y\HfieldComponent_x
\end{bmatrix}, \qquad 
\halfDiv\Hfield := \diver{}\Hfield_{xy}= \partial_x \HfieldComponent_x + \partial_y \HfieldComponent_y.
\end{align*}
In particular, $\halfDiv\Hfield$ depends only on $\Hfield_{xy}$, and is independent of $\HfieldComponent_z$. Therefore, according to Proposition~\ref{Prop:sourceupdate}, it is only required for the in-plane part $\Hfield_{xy}$ to be discretized in a $H(\curl{})$-conforming space for the involution constraint to hold. Thus, we can approximate $\HfieldComponent_z$ in the standard $\mathcal{C}^0$ Lagrange space $\FESpaceHypComp$. Note that although $\halfCurl\Hfield$ involves only $\HfieldComponent_{z}$ in its in-plane components, the nonlinear Hall terms $\halfCurl\Hfield \times \Hfield$ couple all three components of $\Hfield$ nontrivially through the cross product, and must be assembled using the full three-component field. We therefore define the joint magnetic field spaces
\begin{align*}
    \FESpaceH^{2.5} &= \FESpaceH^2 \oplus \FESpaceHypComp
    = \big\{ \Htest_h = (\Htest_{xy,h}, \HtestComponent_{z,h}) \ \big| \ \Htest_{xy,h}\in\FESpaceH^2, \ \HtestComponent_{z,h}\in\FESpaceHypComp \big\}, \\
    \FESpaceHtangent^{2.5} &= \big\{ \Htest_h \in \FESpaceH^{2.5} \, \big| \, \Htest_h \times
\normal = 0 \text{ on } \partial\domain \big\},
\end{align*}
here $\FESpaceH^2$ denotes the two-dimensional BDM space, see Section
\ref{sec:SpaceDisc}. The weak formulation is then posed as: Find
$\{\vel_h^{n+1}, \Hfield_h^{n+1}\} \in
\mathbb{V}_h^3 \times
\FESpaceHtangent^{2.5}$ such that
\begin{align*}
\left\{
\begin{aligned}
&\langle \rho_h^n (\vel_h^{n+1} - \vel_h^n), \veltest_h  \rangle
- \dt_n \mu ( (\halfCurl\Hfield_h^{n+\frac{1}{2}} \times
\Hfield_h^{n+\frac{1}{2}})
, \veltest_h )_{\Ltwo}
= \bzero \\
&\mu (\Hfield_h^{n+1} - \Hfield_h^n , \Htest_h)_{\Ltwo}
+ \dt_n \mu ((\halfCurl\Htest_h \times \Hfield_h^{n+\frac{1}{2}}) ,
\vel_h^{n+\frac{1}{2}} )_{\Ltwo} \\
& \ \ \ + \dt_n ( \resist_h^n \halfCurl\Hfield_h^{n+\frac{1}{2}} ,
\halfCurl\Htest_h)_{\Ltwo}
\\
&\ \ \ + \dt_n \mu d_i \big(\tfrac{1}{\rho_h^n}
(\halfCurl\Hfield_h^{n+\frac{1}{2}}
\times \Hfield_h^{n+\frac{1}{2}}), \halfCurl\Htest_h \big)_{\Ltwo}
= \bzero
\end{aligned}
\right.
\end{align*}
for all $\{\veltest_h, \Htest_h\} \in \mathbb{V}_h^3 \times
\FESpaceHtangent^{2.5}$.

% \section{Hyperbolic solver}\label{sec:HypSolver}

\bibliographystyle{plain}
\bibliography{biblio}

\end{document}